\documentclass{article}
\usepackage{amsmath,amssymb,color,amsthm,hieroglf,bbm,bm}
\usepackage{cjhebrew}
\usepackage{graphicx}
\usepackage{epsfig}
\usepackage{fullpage}
\usepackage{cite}
\usepackage{palatino}
\usepackage{soul}
\numberwithin{equation}{section}
\usepackage{comment}
\usepackage[normalem]{ulem}
\newtheorem{thm}{Theorem}[section]
\newtheorem{lem}[thm]{Lemma}
\newtheorem{prop}[thm]{Proposition}

\newtheorem{define}[thm]{Definition}
\newtheorem{definition}[thm]{Definition}
\newtheorem{notation}[thm]{Notation}

\newtheorem{example}[thm]{Example}

\newtheorem{remark}[thm]{Remark}

\newcommand{\R}{\mathbb{R}}

\numberwithin{equation}{section}

\DeclareMathOperator{\conv}{Conv}
\DeclareMathOperator{\vol}{Vol}

\DeclareMathOperator{\Sym}{Sym}

\newcommand{\av}[1]{\left|{#1}\right|}
\newcommand{\ip}[2]{\left\langle{{#1}},{{#2}}\right\rangle}
\newcommand{\norm}[1]{\left\|{#1}\right\|}

\newcommand{\be}{\begin{enumerate}}
\newcommand{\bi}{\begin{itemize}}
\newcommand{\ee}{\end{enumerate}}
\newcommand{\ei}{\end{itemize}}
\newcommand{\ii}{\item}

\renewcommand{\P}{\mathbb{P}}
\newcommand{\E}{\mathbb{E}}

\renewcommand{\phi}{\varphi}

\newcommand{\psync}{\mathcal{P}_{{{\mathsf{sync}}}}}
\newcommand{\pfsync}{\mathcal{P}^{(F)}_{{{\mathsf{sync}}}}}

\newcommand{\bdy}{\partial}

\newcommand{\mc}[1]{\mathcal{#1}}

\newcommand{\rcs}[1]{R^{\mathsf{CS}}_{#1}}
\newcommand{\rdb}[1]{R^{\mathsf{DB}}_{#1}}
\newcommand{\ics}[1]{\inscr^{\mathsf{CS}}_{#1}}
\newcommand{\idb}[1]{\inscr^{\mathsf{DB}}_{#1}}
\newcommand{\ccs}[1]{\circscr^{\mathsf{CS}}_{#1}}
\newcommand{\cdb}[1]{\circscr^{\mathsf{DB}}_{#1}}

\newcommand{\inscr}{\mathcal{I}}
\newcommand{\circscr}{\mathcal{C}}

\newcommand{\circnorm}[2]{\left\|{#1}\right\|_{\circscr({#2})}}
\newcommand{\inscnorm}[2]{\left\|{#1}\right\|_{\inscr({#2})}}

\newcommand{\ymax}{y_{\mathsf{max}}}
\newcommand{\ymin}{y_{\mathsf{min}}}
\newcommand{\ymaxx}[1]{y_{\mathsf{max},{#1}}}
\newcommand{\yminn}[1]{y_{\mathsf{min},{#1}}}

\newcommand{\proj}{{\mathsf{P}^{(N)}}}
\newcommand{\snorm}[1]{\left\|{#1}\right\|_{\mathsf{spr}}}
\newcommand{\msamp}{{M_{\mathsf{samp}}}}
\newcommand{\tncn}{\tau_N\mc{Q}_N}

\begin{document}

\title{Synchronization conditions in the Kuramoto model and their relationship to seminorms}
\date{\today}

\author{Jared C. Bronski \\ University of Illinois \and Thomas E. Carty
  \\ Bradley University \and
Lee DeVille \\
  University of Illinois }

\maketitle
\begin{abstract}
  In this paper we address two questions about the synchronization of
  coupled oscillators in the Kuramoto model with all-to-all coupling. In the first part we use
  some classical results in convex geometry to prove bounds on the size of the
  frequency set supporting the existence of stable, phase locked
  solutions and show that the set of such frequencies can be expressed by a seminorm which we call the Kuramoto norm. In the second part we use some ideas from extreme order
  statistics to compute upper and lower bounds on the probability of
  synchronization for very general frequency distributions. We do so
  by computing exactly the limiting extreme value distribution of a
  quantity that is equivalent to the Kuramoto norm.
\end{abstract}

{\bf Keywords:} Kuramoto model, convex analysis, permutahedron, extreme-value statistics

{\bf AMS subject classifications:} 34C15, 34D06, 52A20, 60F17\\

\section{Introduction}

In this paper we consider the Kuramoto model of coupled oscillators with homogeneous coupling, i.e. the system of equations typically posed in the form
\begin{equation}\label{eq:K1}
  \frac{d}{dt}\theta_i =  \omega_i + \frac K N \sum_{j=1}^N \sin(\theta_j - \theta_i).
\end{equation}

Since the model's inception more than 40 years ago, it has been found to be  useful in a wide variety of practical applications, including   neuronal networks, Josephson junctions and laser arrays, chemical oscillators, charge density waves, control theory, and electric power networks.  We direct the interested reader to the following survey papers~\cite{S, Acebron.etal.05, Dorfler.Bullo.14, rodrigues2016kuramoto}.

The vector $\omega = (\omega_1,\dots,\omega_N)$ is called the frequency vector.  The topic of interest in this paper is the geometry of the set
of frequency vectors for which~\eqref{eq:K1} supports stable completely phase-locked solutions.  By rescaling the frequency vector we can set the coupling coefficient to unity, and so the equation we study in this paper is
\begin{equation}\label{eq:K}
   \frac{d}{dt}\theta_i =  \omega_i +  \sum_{j=1}^N \sin(\theta_j - \theta_i).
\end{equation}

There has been a great deal of interest in developing both necessary and sufficient analytical conditions for the existence of a stable phase-locked state in this and closely related systems~\cite{Peskin.75, Kuramoto.75, Sastry.Varaiya.80, Sastry.Varaiya.81, Ermentrout.1985,S, Kuramoto.91, Kuramoto.book,PRK.book, C,  CD, MS2, HaHaKim, Chopra.Spong.2009, VM1, VM2, Pecora, AeyelsRogge, deSmetAeyels, Cows, SM1, Acebron.etal.05, Mirollo.Strogatz.05,  ashwin2006extreme, Wiley.Strogatz.Girvan.06, Abrams.Mirollo.Strogatz.Wiley.08, Arenas.etal.08, Dorfler.Bullo.2011, Dorfler.Bullo.12, Dorfler.Chertkov.Bullo.13, Dorfler.Bullo.14, rodrigues2016kuramoto, Bronski.DeVille.Ferguson.16, Delabays.Coletta.Jacquod.16, Delabays.Coletta.Jacquod.17, troy2017phaselocked, Bronski.Ferguson.2018,Ferg,Ferg.18}.  We are inspired here by two particular prior results in the literature. The first is a sufficient condition for phase locking due to D\"{o}rfler and Bullo\cite{Dorfler.Bullo.2011} that says that (in this scaling)
a sufficient condition for full phase-locking is that
\begin{equation}
\max_{i,j} (\omega_i-\omega_j)< N.
\label{eqn:DB}
\end{equation}
The second result is due to Chopra and Spong\cite{Chopra.Spong.2009} , and says that a necessary condition for full phase-locking is that
\begin{equation}
\max_{i,j} (\omega_i-\omega_j) <\frac{1}{8\sqrt{2}} \left(\sqrt{32 + (N-2)^2} + 3 (N-2)\right)\sqrt{16 + (N-2) \sqrt{32 + (N-2)^2} - (N-2)^2}.
\label{eqn:CS}
\end{equation}
\begin{remark}
We note that $\max_{i,j}(\omega_i-\omega_j) = \max_{i,j}\av{\omega_i-\omega_j}$, and so we can think of the quantities in the last two equations as  a kind of mean-adjusted version of the $L_\infty$ norm.  However, as the reader will see below, it is useful to write this without the absolute values when we interpret the geometry of these stable sets.
\end{remark}

In this paper we use some ideas from convex geometry to
better understand the geometry of the set of frequencies supporting stable phase-locking.
The basic observation is as follows: suppose that we have a bounded
convex region $\Omega$ along with $R$, a set of points on the boundary
$\partial \Omega$. From such a collection of points we can construct
two polytopes. The first polytope, which is inscribed in $\Omega$ and which we denote $P_R$, is the convex hull of the points in
$R$. Since the points in $R$ lie on the boundary of a convex region this is a
polytope with vertices given by the points in $R$, and this
polytope is contained in $\Omega$.   The
second polytope, which circumscribes $\Omega$ and which we denote $Q_R$,  is constructed by taking the intersection over all
points in ${\bf x}\in R$ of the supporting  half-space to
$\Omega$ at ${\bf x}\in R$, and thus this polytope contains $\Omega$.

An example of these constructions is illustrated in Figure \ref{fig:BasicIdea} (a--c). Figure  \ref{fig:BasicIdea} (a)
depicts a convex region $\Omega$ together with a collection of points
$R$ lying on the boundary of the region. Figure  \ref{fig:BasicIdea}
(b) shows $P_R$ and Figure  \ref{fig:BasicIdea} (c)
depicts $Q_R$ --- note that $P_R$ is inscribed inside $\Omega$ and $Q_R$ circumscribes $\Omega$.

\begin{figure}[!tbp]
\centering
\hspace{-4mm}
  \begin{minipage}[b]{0.3\textwidth}
    \centering
    \includegraphics[width=\textwidth]{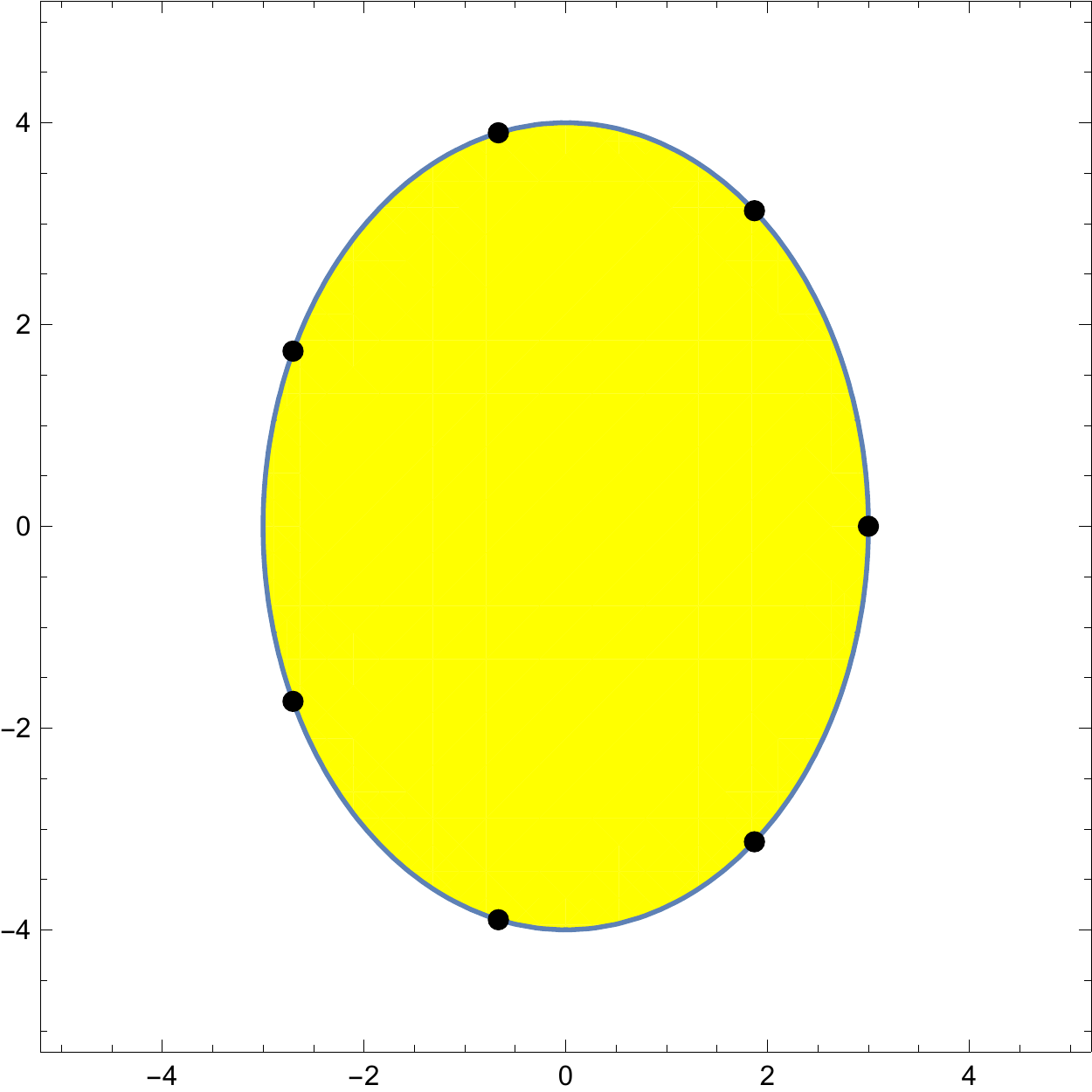}\\
    (a) Convex Region $\Omega$ and points $R\subseteq \bdy \Omega$.
  \end{minipage}
  \begin{minipage}[b]{0.3\textwidth}
     \centering
    \includegraphics[width=\textwidth]{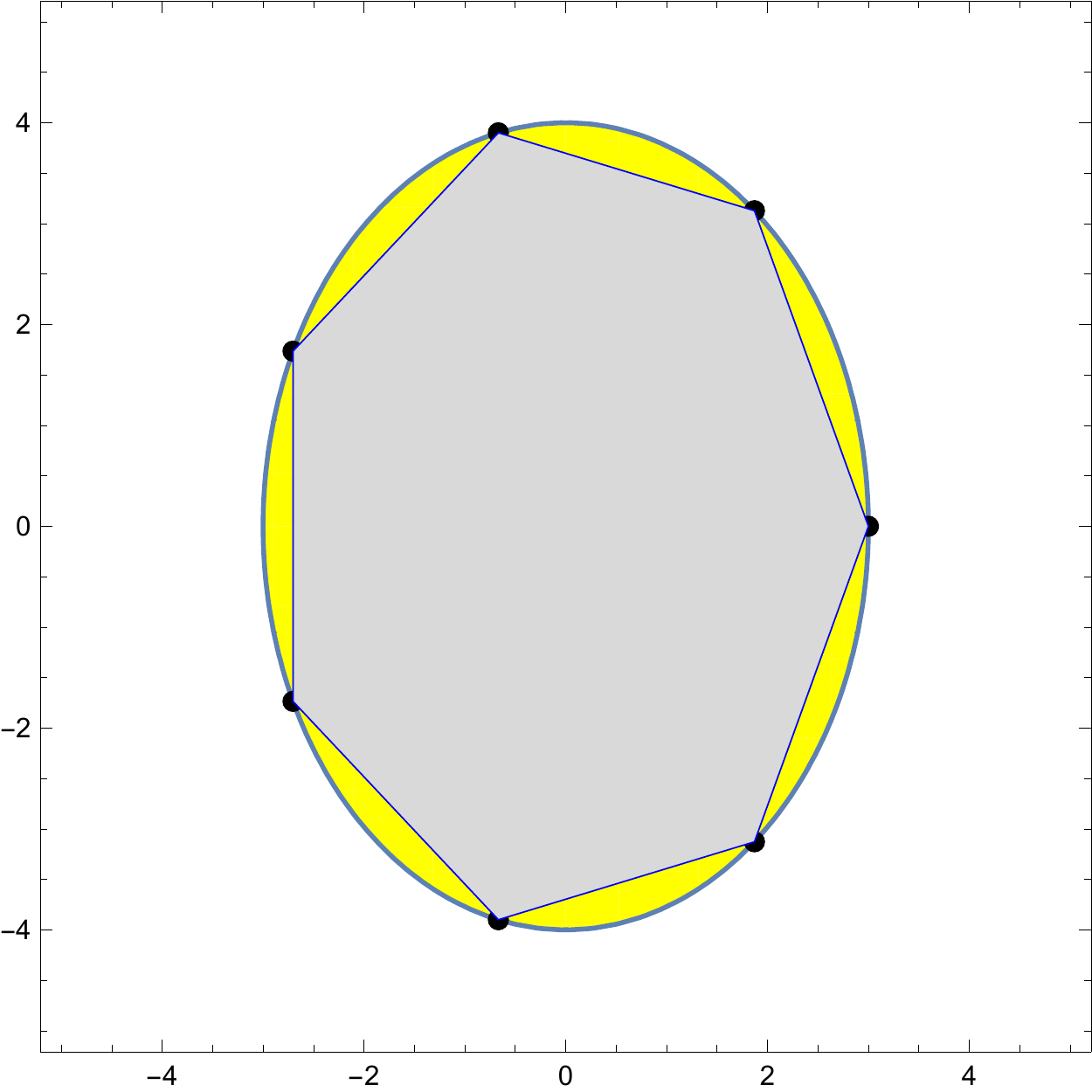}\\
    (b) Convex Region $\Omega$ and inscribed polytope $\mc{I}_R$.
  \end{minipage}
  \begin{minipage}[b]{0.3\textwidth}
     \centering
    \includegraphics[width=\textwidth]{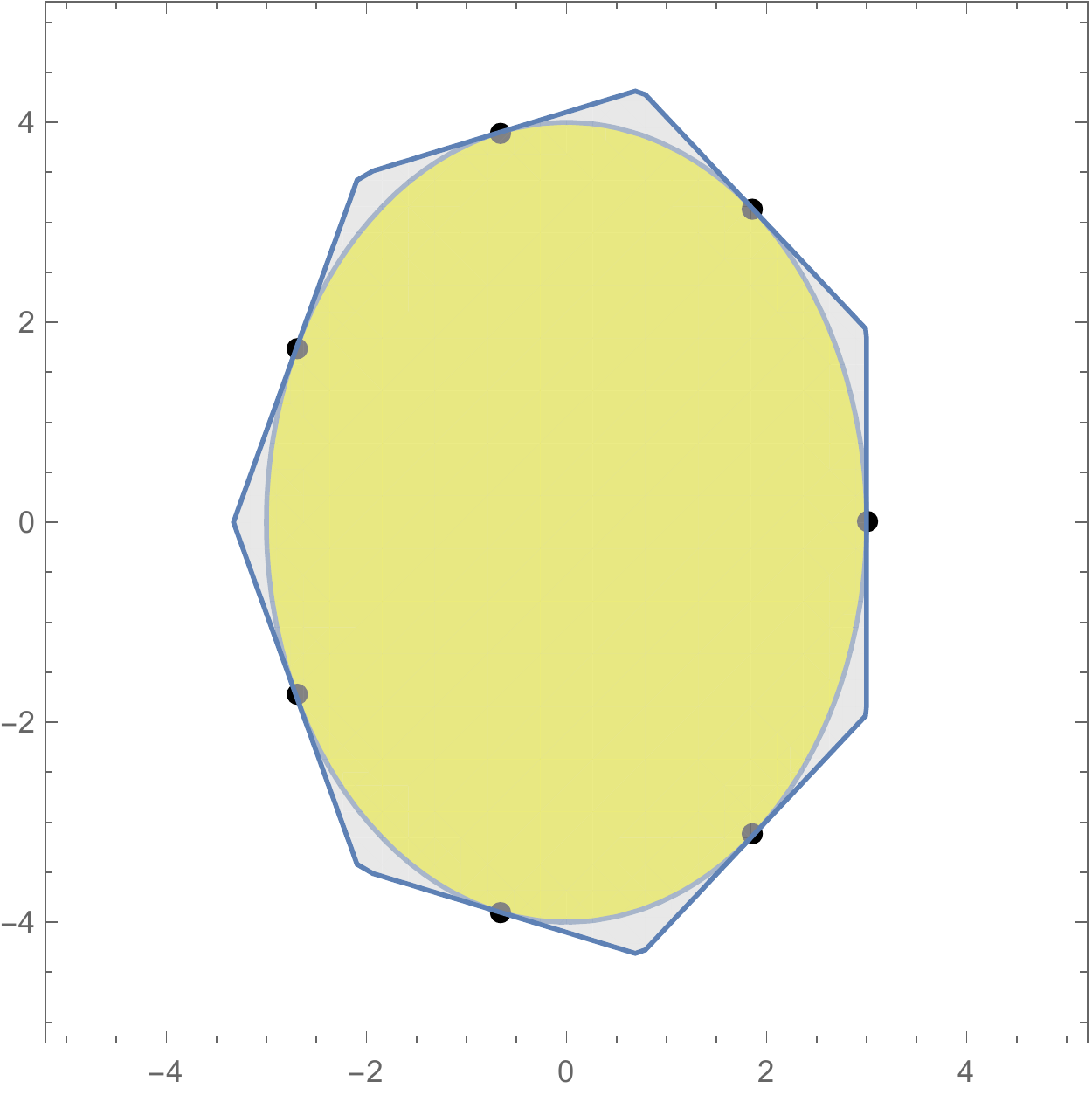}\\
    (c) Convex Region $\Omega$ and circumscribed polytope $\mc{C}_R$.
  \end{minipage}
\label{fig:BasicIdea}
\caption{A convex Region $\Omega$ and set of boundary points $R$
  together with the corresponding inscribed and circumscribed polytopes.  Note that these can be thought of as necessary and sufficient conditions, respectively. }
\end{figure}

We will give constructions for several different sets of points $R$ together with the polytopes $P_R,Q_R$, with each collection of points implying a sufficient $(P_R$) and necessary ($Q_R$) condition for the existence of a phase-locked solution. We also show that these conditions can be expressed in terms of some norm of the frequency vector $\omega$. We want to stress an important point here:  the description of a polytope in terms of a norm can have significant advantages over its combinatorial description; in particular, it is in many cases much more computationally efficient to check whether it contains a particular vector.

The D\"{o}rfler--Bullo sufficient condition (\ref{eqn:DB}) and the Chopra--Spong necessary condition (\ref{eqn:CS}) each correspond to different choices of point sets $R,$ and so our construction gives two new conditions, a necessary condition which is dual to the D\"{o}rfler--Bullo sufficient condition, and a sufficient condition that is dual to the Chopra--Spong necessary condition. These conditions can each be expressed as some norm of the frequency vector $\omega$.

We also show how to combine two  sufficient (resp. necessary) conditions to get a new condition that is
strictly better in the sense that it contains both sufficient conditions (resp. is contained in the intersection of both necessary conditions).
This procedure allows us to establish new conditions that are better than any existing in the literature, in two senses:  we first show that any of the existing conditions can be ``dualized'' in a natural way, and then we show that any two existing bounds can be combined to give conditions strictly stronger than each of them.

\section{Convex geometry and the size of the stable phase-locked region}

\subsection{Background and Previous Results}
Throughout this paper we consider the Kuramoto model mentioned in the introduction:
\begin{equation}\tag{\ref{eq:K}}
\frac{d\theta_i}{dt} =  \omega_i + \sum_{j=1}^N \sin(\theta_j-\theta_i),\quad i=1,\dots, N.
\end{equation}

\begin{define}
  Let $\R^N_0 = \{x\in \R^N:\sum_{i=1}^n x_i = 0\}$.  Then $\mc D_N = \{\omega\in \R^N_0: ~\eqref{eq:K} \mbox{ has a stable fixed point}\}$.
\end{define}

The study of $\mc D_N$ is the central goal of this paper.    Due to the antisymmetry of the nonlinear term, the sum $\sum_{i=1}^n \theta_i$ precesses around the unit circle with velocity $\overline\omega = \sum_{i=1}^N \omega_i$.  We can always work in the co-rotating frame (i.e. shift by average frequency $\overline\omega/N$) and assume without loss of generality that $\overline\omega=0$.  Conversely, if $\overline\omega\neq0$ then~\eqref{eq:K} will not have a fixed point, but it can have a stable configuration that precesses around the circle with rate $\overline\omega$.

In previous work the following lemma was established:

\begin{lem}\label{lem:stability}
A stationary configuration
of oscillators is stable if and only if the following two conditions
are met:
\begin{enumerate}
\item The  quantities $\kappa_j := \sum_i\cos(\theta_j-\theta_i) >0$ for all $j$;
\item The quantity $\tau := \sum_j\left({\sum_i \cos(\theta_j-\theta_i)}\right)^{-1}<2$.
\end{enumerate}
If these conditions are met the configuration is orbitally stable, with a
single zero eigenvalue arising from the rotational invariance and $N-1$
eigenvalues which are strictly negative.
\end{lem}

\begin{remark} \label{rem:ermentrout}This was Lemma 2.4 in~\cite{BDP}.  The authors also used this characterization of $\mc D_N$ to show that it was convex.  We would also like to contrast the results of Lemma~\ref{lem:stability} with those existing in the literature, specifically~\cite[Theorem 3.1]{Ermentrout.1992}, which gives a sufficient condition for stability in terms of the signs of entries of the Jacobian.  See Section~\ref{sec:example-negative} below for more discussion.

\end{remark}
It also straightforward to see that $\mc D_N$ satisfies
the following additional properties:
\begin{itemize}
\item If ${\bf \omega }{\ \in\R^N_0}$ is sufficiently small then ${\bf \omega} \in \mc D_N$. This follows from an implicit function argument in a neighborhood of ${\bf \omega}=0$:  if $\omega = 0$, then $\theta=0$ is an orbitally stable fixed point;
\item If ${\bf \omega}\in \mc D_N$ then $-{\bf \omega}\in \mc D_N$. This follows from the fact that the Kuramoto model is invariant under the transformation ${\bf \theta} \mapsto -{\bf \theta}, {\bf\omega }\mapsto-{\bf \omega}$.
\end{itemize}

We now state a classical result that is essential below, but first a definition:

\begin{define}\label{def:balanced-absorbing}
  Let $B$ be a subset of a vector space. We say that $B$ is {\bf balanced}, or {\bf symmetric}, if $x\in B\implies -x\in B$, and we say that $B$ is {\bf absorbing} if for any $x\in B$, there exists $\lambda>0$ such that $\lambda x\in B$.
\end{define}

\begin{thm}\label{thm:simon}\cite[Corollary 1.10]{Simon.2011}
Let $B$ be a subset of a vector space.  If $B$ is open (resp.~closed), convex, balanced, and absorbing,
then $B$ is the open (resp.~closed) unit ball of some seminorm.  In fact, one can be a bit more explicit:  if we define
\begin{equation*}
  \norm{x}_B := \inf\{\lambda\in(0,\infty):\lambda^{-1} x \in B\}
\end{equation*}
then $\norm x_B$ is that seminorm.
\end{thm}

Thus there exists some semi-norm $\Vert\cdot\Vert_{\rm Kur}$ with the property that
\begin{equation*}
  \mc D_N = \left\{\omega : \sum_i \omega_i = 0\land \Vert{\bf \omega}\Vert_{\rm Kur}<1\right\}.
\end{equation*}
Note that $\R^N_0$ is a vector space and thus Theorem~\ref{thm:simon} applies with $\R^N_0$ as the ambient space.
It seems unlikely that this norm can be expressed in a simple form in terms of ${\bf \omega}$. However,~\cite{BDP} gave constructions for several polytopes that are contained in $\mc D_N$, giving sufficient conditions for stability. In this section we will show that these polytopes can be realized as the units balls for various norms, and that these norms can be expressed explicitly in terms of ${\bf \omega}$. More importantly we show how, given two necessary or sufficient conditions for stable phase-locking we can combine them to produce a better such condition.

{\begin{remark}
  We will introduce several norms below, but we find the Euclidean norm useful.  Throughout the paper, whenever we refer to a norm without subscripts, this will be the Euclidean norm, or the induced Euclidean norm.
\end{remark}}

\subsection{Constructing Boundary Points}

We begin by giving constructions for several sets of frequency vectors that lie
on the boundary of the phase-locked region, as well as the
corresponding configurations of oscillators. To motivate these
constructions we first note that the frequencies of a phase-locked
state can be determined from the configuration angles $\theta_i$ via
\[
\omega_i = -  \sum_j \sin(\theta_j - \theta_i).
\]
This follows from setting the righthand side of \ref{eq:K} to zero,
and we can consider this as giving a map between phase-locked
configurations ${\bf \theta}$ and phase-locked frequencies ${\boldsymbol{
  \omega}}$. The Jacobian of the vector field is then given by $\bf{J} = - \nabla_{\bf
  \theta} \omega$.  Then one has the obvious identity
\[
\nabla_{\bf \theta} \Vert {\bf \omega} \Vert^2 = - 2 {\bf J} {\bf \omega}.
\]
Therefore any critical point of $\Vert {\bf \omega} \Vert^2$, the
squared length of the frequency vector, with respect to the
configuration ${\bf \theta}$, gives a frequency vector ${\bf \omega}$
that lies in the kernel of the Jacobian. These critical points are
candidates for points on the boundary of the stably phase-locked
region, since at any point on the boundary the Jacobian necessarily has a kernel of
dimension two or higher. In practice we will not try to find critical
points with respect to all possible configurations, but will instead
find critical points with respect to certain submanifolds of very
symmetric configurations. One must then, of course, check that these
are in fact critical points of the full problem. With this in mind we
present two families of special configurations that will be important in this paper.

We will find it necessary to consider vectors with repeated terms below, and in various permutations, so we use the following notation:

\begin{notation}\label{not:abc}
   When we write the vector $(a^{(k)},b^{(l)},c^{(m)})^t$, we mean the (column) vector in $\R^{k+l+m}$ with coefficients
  \begin{equation*}
    (\underbrace{a, a, \ldots,a}_{\text{$k$ times}},\underbrace{b,b,\ldots,b}_{\text{l times}},,\underbrace{c,c,\ldots,c}_{\text{m times}})^t,
  \end{equation*}
  and similarly for more or fewer terms.  Given a vector $x\in \R^n$, we define $\Sym(x)$ as the set of all vectors that can be obtained from $x$ by permuting its coefficients.  In particular, the set $\Sym(a^{(k)},b^{(l)},c^{(m)})$ is the set of all vectors with exactly $k$ entries equal to $a$, $l$ entries equal to $b$, and $m$ entries equal to $c$.
\end{notation}

\begin{define}\label{def:DB}
  For each $1 \le k \le N$, let $v_k = ((N-k)^{(k)}, (-k)^{(N-k)})$ and define
  \begin{equation*}
  \rdb N := \bigcup_{k=1}^N \Sym(v^k).
\end{equation*}
In other words, $\rdb N$ is the set of all vectors in $\R^N$ with $k$ entries equal to $(N-k)$ and $(N-k)$ entries equal to $-k$.  Note that $\rdb N \subseteq\R^N_0$.
\end{define}

\begin{define}\label{def:CS}
Let $z = (1,-1,0,0,\dots, 0)$ and define
\begin{equation}\label{eqn:defoftaun}
\tau_N :=  \max_{\phi\in \R} \left[(N-2) \sin(\phi) +  \sin(2 \phi)\right].
\end{equation}
Then
\begin{equation*}
  \rcs N := \Sym(\tau_N z),
\end{equation*}
that is to say, elements of $\rcs N$ are those vectors with one entry equal to $\tau_N$, one entry equal to $-\tau_N$, and the rest zero.
\end{define}

\begin{prop}
  $\av{\rdb N} = 2^N-2$ and $\av{\rcs N} = N(N-1)$.
\end{prop}
\begin{proof}
  Counting $\rcs N$ is easier:  choose one index to be positive and one to be negative, and there are clearly $N(N-1)$ such choices.

  For $\rdb N$, note that each vector is determined by the set of entries that are positive, but we cannot have all entries positive or have all entries negative.  Therefore the number of elements of $\rdb N$ is the number of nonempty proper subsets of $\{1,2,\dots,N\}$.
\end{proof}

We call the constant $\tau_N$ the \emph{Chopra-Spong constant} and it
may be checked that $\rcs N$ is simply the rescaling of the
root vectors in the $A_N^*$ lattice.  With some computation\cite{Chopra.Spong.2009, BDP}, we can compute $\tau_N$ exactly and asymptotically:
\begin{align*}
\tau_N &=  \frac{1}{16\sqrt{2}} \left(\sqrt{32 + (N-2)^2} + 3 (N-2)\right)\sqrt{16 + (N-2) \sqrt{32 + (N-2)^2} - (N-2)^2} \\
&\approx (N-2) + O(N^{-1})\qquad N\gg 1.
\end{align*}
In particular, the exact formula will be useful in some computations below.  We also note that $N-2 \le \tau_N \le N-1$:  if we plug $\phi\mapsto \pi/2$ into the right-hand side of~\eqref{eqn:defoftaun}, we get immediately that $\tau_N\ge N-2$, and clearly $\tau_N \le N-1$.

We can now state the following proposition, which is that both of these special sets of configurations are always contained in the boundary of the set of configurations that give rise to stable solutions, i.e. are always bifurcation points for~\eqref{eq:K}:

\begin{prop}
  We have $\rdb N \subseteq \bdy \mc D_N$ and $\rcs N\subseteq \bdy \mc D_N$.
\end{prop}

First we consider $\rdb N$.  A relatively straightforward computation (for details see Bronski, DeVille and Park\cite{BDP}) that these configurations are fixed points of the Kuramoto flow, and that they lie on the boundary of the region of stability: the Jacobian has a two dimensional kernel spanned by $(1,1,1,\ldots, 1)^t$ and $\omega$.

Now for $\rcs N$. These points were originally constructed by Chopra and
Spong \cite{Chopra.Spong.2009} in the construction of a sharp necessary condition and later, from a somewhat different point of
view, by Bronski, DeVille and Park\cite{BDP}. The basic idea is
to find the configuration admitting the largest possible frequency
difference: that is to say maximizing the quantity
\[
\omega_i - \omega_j = 2 \sin(\theta_i - \theta_j) + \sum_k\left(
\sin(\theta_k-\theta_j) - \sin(\theta_k - \theta_i)
\right)
\]
over all $\theta$. By an application of Lagrange multipliers, we see that a maximizing configuration must have the form
$\theta_i = \phi, \theta_j = -\phi, \theta_k=0$, and maximizing $\phi$ gives $\tau_N$.  We can check directly that this configuration is a fixed point,
and that the Jacobian is positive semi-definite with a two dimensional kernel spanned by $(1,1,1,\ldots,1)^t$ and $\omega$ and this is therefore a stable phase-locked solution (see~\cite{Mirollo.Strogatz.05, Ermentrout.1985}).

\begin{remark}

We will show later in this paper that the sufficient condition implied by this set of points $\rdb N$ is exactly D\"{o}rfler--Bullo condition
\[
\max_{i,j} (\omega_i - \omega_j )\leq N.
\]
There is a complementary necessary condition, also expressible in terms of some explicit norm of the frequency vector $\omega$ that we will compute later in the paper.

One way to think about these configurations is via symmetry. Since all oscillators are identical the phase locked region
must be invariant under the symmetry group $S_N \times S_2$ consisting
of all permutations of the coordinates together with $\omega \mapsto
-\omega.$ In particular if
$\omega$ lies on the boundary of the phase-locked region then any
permutation of $\omega$ must lie on the boundary of the phase-locked
region. For fixed $|S|=\sigma$ these frequency vectors represent configurations that
are invariant under the subgroup $S_\sigma \times S_{N-\sigma}$ of
permutations fixing $S$. If one takes $j$ oscillators to be at angle
$\varphi$ and $N-j$ oscillators to be at angle $0$ then the
corresponding frequency for which this is a fixed point is given by
$\omega =\sin(\phi) ((N-j)^j, (-j)^{N-j})^t$. The length of this vector
has a critical point at $\varphi = \frac{\pi}{2}$. With a bit of extra
work one can check that this critical point with respect to a subset
configurations actually lies in the kernel of ${\bf J}$.

\end{remark}

\begin{remark}

The points in $\rcs N$ are precisely the vertices of the (rescaled) Voronoi
cell for the $A_N$ root lattice -- see the text of Conway and Sloane
\cite{CS} for details.

Another interpretation of  these configurations is as follows: one can think of maximizing
$\Vert \omega \Vert^2_2$ over the subset of configurations that
are invariant under $S_{N-2}$: there are $N-2$ oscillators at the
origin, one oscillator at angle $\phi$ and one at angle $-\phi$. The
$\omega$ corresponding to this configuration is of the form $\omega = ((N-2)
\sin \varphi + \sin 2 \varphi) (1,-1,0^{N-2})^t$. Maximizing the length
of this vector over $\varphi$ leads to the Chopra-Spong constant.
\end{remark}

We can further generalize the Chopra-Spong calculation to define a
family of sets of points on the boundary of the phase-locked region.
\begin{define}
Let $\rcs {N,j}$ consist of all permutations of the vector
\[
\tau_{N,j} \cdot(1^{(j)},(-1)^{(j)}, 0^{(N-2j)})^t,
\]
where the constant $\tau_{N,j}$ is defined as
\[
\tau_{N,j} = \max_\phi [(N-2 j) \sin\phi + j \sin 2\phi].
\]
\end{define}

These frequencies represent configurations of the following form:
There are $N-2j$ oscillators at angle zero, $j$ oscillators that lead
this group by angle $\phi^*$, where $\phi^*$ is the argument that
maximizes the quantity $\tau_{N,j}$, and $j$ oscillators which trail
this group by angle $-\phi^*$.   It may be verified that the constant
$\tau_{N,j}$ can be computed explicitly as
\[
\tau_{N,j} = \frac{\left(\sqrt{36 j^2-4 j N+N^2}-6 j+3 N\right) \sqrt{-2 \sqrt{36 j^2-4 j
   N+N^2}+\frac{N \left(\sqrt{36 j^2-4 j N+N^2}-N\right)}{j}+12 j+4 N}}{16 \sqrt{2 j}},
\]
and that the case $j=1$ reduces to the Chopra-Spong constant.  A couple of things to note here. Firstly we observe that  these generalized Chopra-Spong points lie on the boundary of the phase-locked region $\partial {\mathcal D}_N$. Secondly note that the argument above gives us that $N-2j \le \tau_{N,j} \le N-j$. Finally notice that these points are well-defined for $2j\leq N$, and that for $N$ even and $j=N/2$ these points are a strict subset of the D\"orfler-Bullo points.

\subsection{The inscribing and circumscribing polytopes}

In this section, we will present a method that takes a set of points and generates two special polytopes.  We will show that when these points are chosen to lie on the boundary of any convex set $\Omega$, one of these polytopes will be inscribed inside $\Omega$, and the other will circumscribe $\Omega$.  In the previous section, we presented two natural collections of points living on the boundary of the phase-locked region $\mc{D}_N$; putting this together will lead to two inscribing and two circumscribing polytopes for the phase-locked region $\mc{D}_N$.

\begin{define}
Given any finite  collection of points $R$ we define
two polytopes $\inscr(R), \circscr(R)$. \bi\ii The polytope $\inscr(R)$ is defined as the convex
hull of the points in $R$. {For the cases of interest here all of the points in $R$ are extremal and the convex hull of $R$ is the polytope with vertices given by the elements of $R$.} \ii The polytope $\circscr(R)$ is defined as
follows: given a point $x \in R$ define the supporting half-space
$H_x$ as the closed half-space containing the origin whose boundary
$\partial H_x$ has normal vector $x$. Let $\circscr(R)$ be the intersection of these
supporting half-spaces
\[
\circscr(R) = \bigcap_{x\in R} H_x.
\]
\ei

\end{define}

\begin{remark}
  It follows easily that if $\Omega$ is convex and $R\subseteq \partial\Omega$, then we have the
inclusions
\[
\inscr(R) \subseteq \Omega \subseteq \circscr(R).
\]
We remark that convexity of the phase-locked region is only known for the  case of equally weighted all-to-all
coupling, and the methods used here are only applicable when there is a convex phase-locked domain.
\end{remark}

\begin{notation}
  Since the sets $\rcs N, \rdb N$ figure so prominently in the sequel, we will simplify notation slightly by writing
  \begin{equation*}
    \ics N:= \inscr(\rcs N),\quad
    \idb N:= \inscr(\rdb N),\quad
    \ccs N:= \circscr(\rcs N),\quad
    \cdb N:= \circscr(\rdb N).
  \end{equation*}
  
\end{notation}

\begin{remark}
We now give a few examples, but one note on visualization.  For any given $N$, we can represent $\rcs N, \rdb N$ as living in $\R^{N-1}$ after we have chosen a basis for $\R^N_0$. We will make the following choices below:  for any $N$ and $1\le i\le N$, we define
\begin{equation*}
  w_{N,k} = ((1)^{(k)}, (-k)^{(1)}, (0)^{(N-k-1)})^t,
\end{equation*}
and let $u_{N,k} = w_{N,k}/\norm{w_{N,k}}$. As an example, we will represent a generic vector in $\R^4_0$ by
\begin{equation*}
  x\begin{pmatrix}
    1/\sqrt2\\-1\sqrt2\\0\\0
  \end{pmatrix} +
  y\begin{pmatrix}
    1/\sqrt6\\1/\sqrt6\\-2\sqrt6\\0
  \end{pmatrix}+
  z\begin{pmatrix}
    1/\sqrt{12}\\1\sqrt{12}\\1\sqrt{12}\\-3\sqrt{12}
  \end{pmatrix}.
\end{equation*}

\end{remark}

\begin{example}
For $N=3$, the set $\rdb3$ consists of the six vectors
\[
\pm \left(\begin{array}{c}2 \\ -1 \\ -1 \end{array}\right),  \pm \left(\begin{array}{c}-1 \\ 2 \\ -1 \end{array}\right),  \pm \left(\begin{array}{c}-1 \\ -1 \\ 2 \end{array}\right)
\]
while the set $\rcs3$ consists of the six vectors
\[
\pm \tau_3 \left(\begin{array}{c}1 \\ -1 \\ 0 \end{array}\right), \pm \tau_3 \left(\begin{array}{c}1 \\ 0 \\ -1 \end{array}\right), \pm \tau_3 \left(\begin{array}{c}0 \\ 1 \\ -1 \end{array}\right),
\]
where the Chopra-Spong constant is $\tau_3 =
\frac{\left(3+\sqrt{33}\right)\sqrt{15+\sqrt{33}}} {16
  \sqrt{2}}\approx 1.76017$.
In this case $\idb3$, $\cdb3$, $\ics3$, and  $\ccs3$
are all regular hexagons of side lengths $\sqrt{6},
2\sqrt{2}, \tau_3\sqrt{2} , \tau_3\frac{2\sqrt{6}}{3} $
respectively. $\idb3$ and $\ccs3$ are oriented the same
way, as are $\cdb3$ and $\ics3$, and the two pairs are
offset from one another by $\frac{\pi}{6}$. One can get some sense of
the tightness of these inclusions by computing the areas of these
hexagons. We have $\av{\idb 3} = 9\sqrt{3} \approx 15.5885$, $\av{\cdb3} = 12\sqrt{3} \approx 20.7846$, $\av{\ics 3} = 3
\sqrt{3} (\tau_3)^2 \approx 16.0987$ and
$\av{\ccs3}=4\sqrt{3} (\tau_3)^2 \approx 21.4649.$
Since the $\inscr$ polytopes are contained in the phase-locked region and
the $\circscr$ polytopes contain the phase-locked region this gives us upper
and lower bounds on the true area of $20.78$ and $16.09$
respectively. We note that while the $\rcs 3$ points give the
better inner approximation and the $\rdb 3$ points give the better outer
approximation in terms of area there are regions which are contained
in one which are not contained in the other.  (See Figure \ref{fig:HexagonsOfThreeOscillator}.)
\begin{figure}[!tbp]
\centering
  \begin{minipage}[b]{0.45\textwidth}
\centering
    \includegraphics[width=\textwidth]{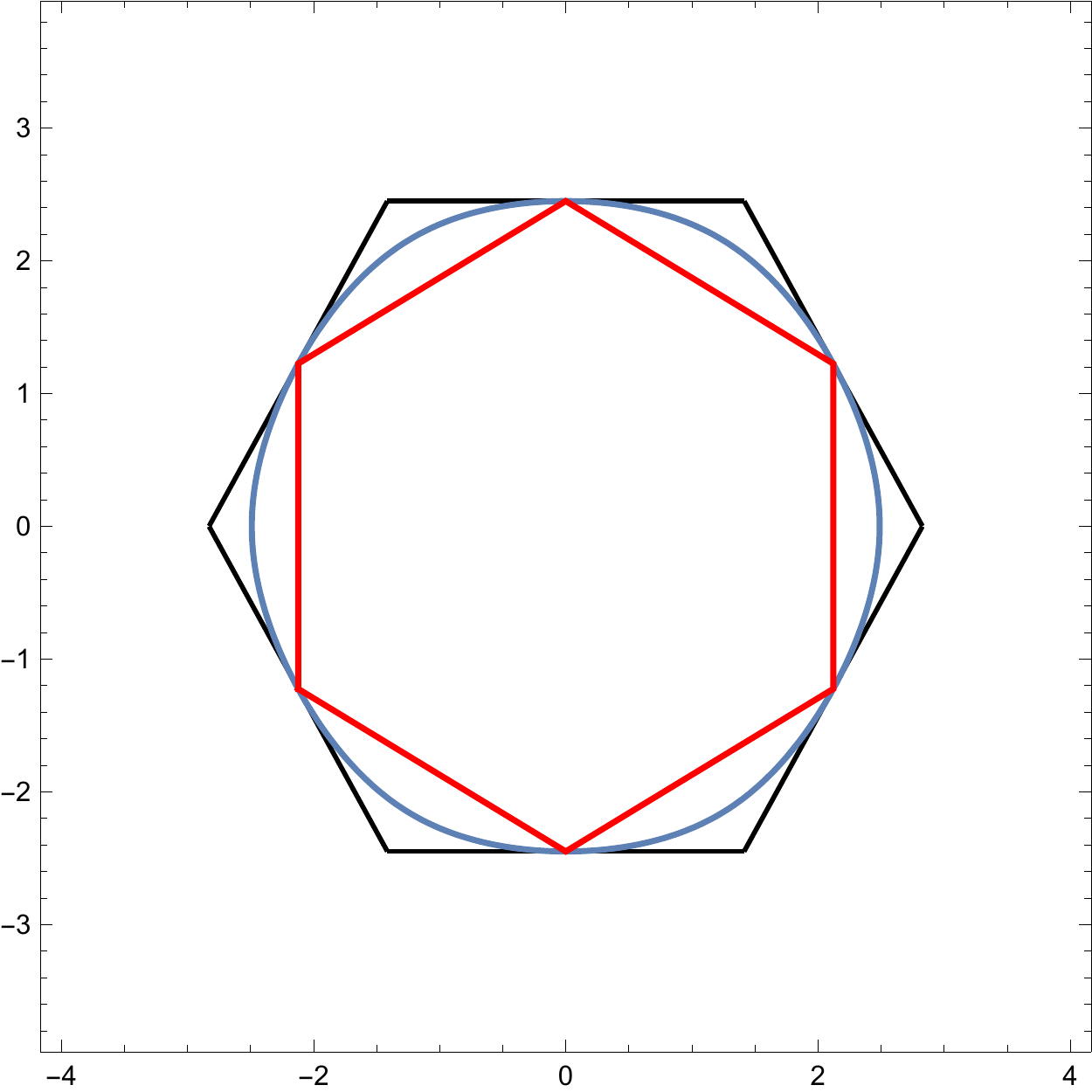}\\
    (a) The necessary (black) and sufficient (red) conditions defined by $R_3^{\text{DB}}$
 \end{minipage}
  \begin{minipage}[b]{0.45\textwidth}
\centering
    \includegraphics[width=\textwidth]{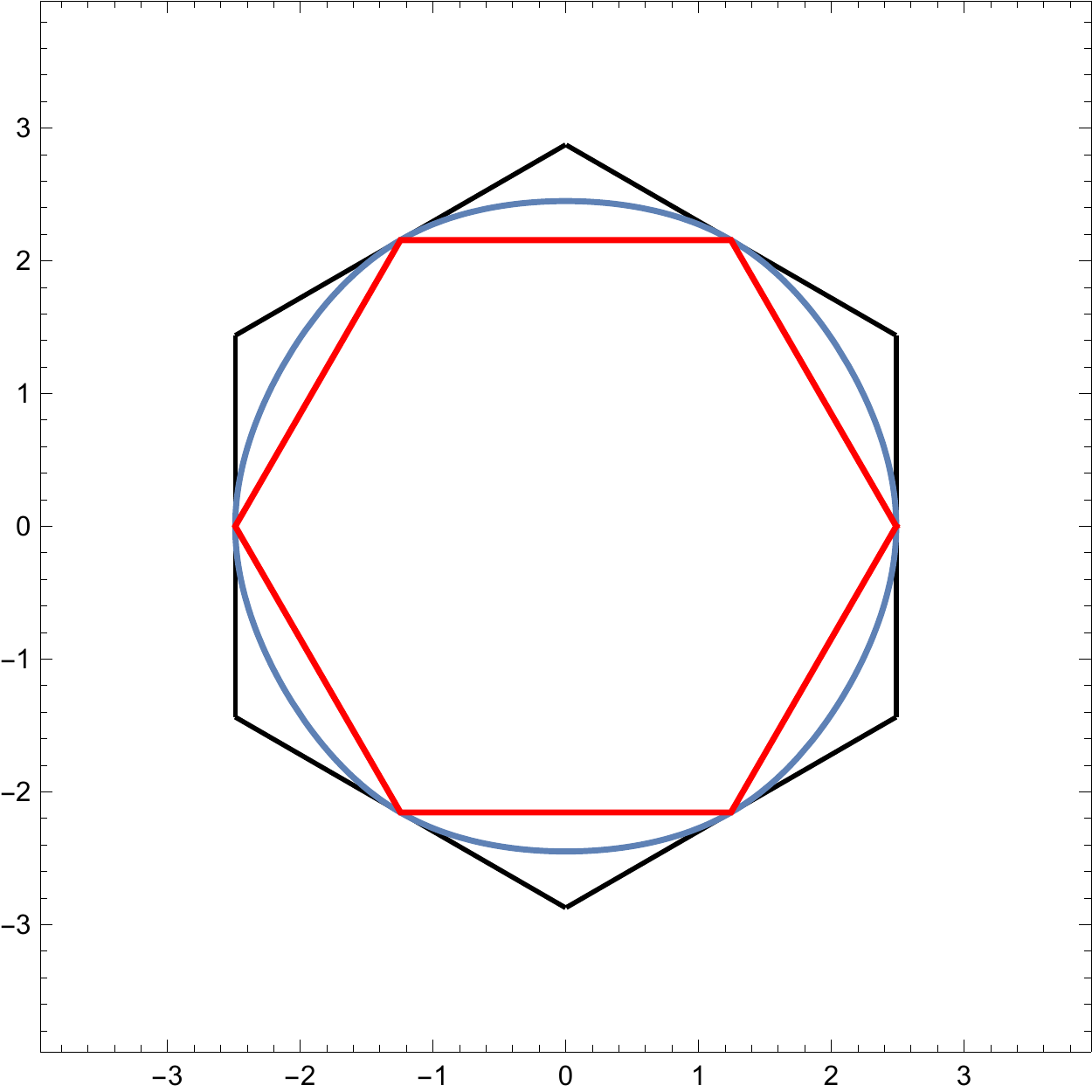}\\
    (b)  The necessary (black) and sufficient (red) conditions defined by$R^{\text{CS}}_3$
\end{minipage}
\caption{The phase-locked region and the necessary and sufficient conditions
  defined to the points $R^\text{DB}_3$ and $R^\text{CS}_3$.  {In each case, we are using red to denote the $\inscr$ polytopes and black to represent the $\circscr$ polytopes.}}
\label{fig:HexagonsOfThreeOscillator}
\end{figure}

\end{example}

In the case $N=3$, the two dimensional polyhedra ${\mathcal I}_3^{\text{CS}}, {\mathcal I}_3^{\text{DB}},{\mathcal C}_3^{\text{CS}},{\mathcal C}_3^{\text{DB}}$ are hexagons. This is a bit misleading, since the higher dimensional polytopes
are much richer.  We get a glimpse of this when we consider $N=4$.

\begin{example}
For $N=4$ the set $\rdb4$ consists of the fourteen vectors given by
all permutations of
\[
\pm \left(\begin{array}{c}3 \\ -1 \\ -1 \\ -1 \end{array}\right),  \left(\begin{array}{c}2 \\ 2 \\ -2 \\ -2 \end{array}\right).
\]
The set $\rcs4$ consists of the twelve vectors given by all
permutations of
\[
\tau_{4} \left( \begin{array}{c} 1 \\ -1 \\ 0 \\ 0 \end{array}\right)
\]
where the Chopra--Spong constant is $\tau_4 =
\frac{3\sqrt{3}}{2}$.
The polytope $\idb4$ is a rhombic dodecahedron ($V=14,E=24,F=12$)
with edge length $2\sqrt{3}$ and volume 128. The polytope
$\cdb4$ is a truncated octahedron ($V=24,E=36,F=14$) with edge
length $2\sqrt{2}$ and volume 256. The
polytope $\ics4$ is a cuboctahedron ($V=12,E=24,F=14$) with
edge length $\tau_4^{\text{CS}} \sqrt{2} = \frac{3 \sqrt{6}}{2}$ and volume
$\frac{135\sqrt{3}}{2}\approx 117$. The
polytope $\ccs4$ is a rhombic dodecahedron with the same
orientation as   $\idb4$. The
polytope $\ccs4$ has volume $162\sqrt{3}\approx 280.6$. The volume of the actual phase-locked
region is approximately $210.$, via numerical integration.

The inscribed polyhedra $\inscr$ are depicted in Figures \ref{fig:inscribed4}a
and~\ref{fig:inscribed4}b. It is clear from the graphs that these figures are
dual polytopes --- the vertices of one are (up to scaling) the perpendiculars to the faces of the
other.  One can also see that, while the volume of the inscribed
rhombic dodecahedron is somewhat larger than that of the inscribed
cuboctahedron the two are not strictly comparable - there are points
in each set that are not contained in the other.

Similarly the circumscribed polyhedra are depicted in Figures
\ref{fig:circumscribed4}a and
\ref{fig:circumscribed4}b. Again we see that the volume of the circumscribed
rhombic dodecahedron (the polytope whose normals are given by the
Chopra-Spong points) has a somewhat larger volume than that of the
circumscribed truncated octahedron, but again there are points in each
set which are not contained in the other.

\end{example}

\begin{figure}[!tbp]
\centering
  \begin{minipage}[b]{0.45\textwidth}
\centering
    \includegraphics[width=\textwidth]{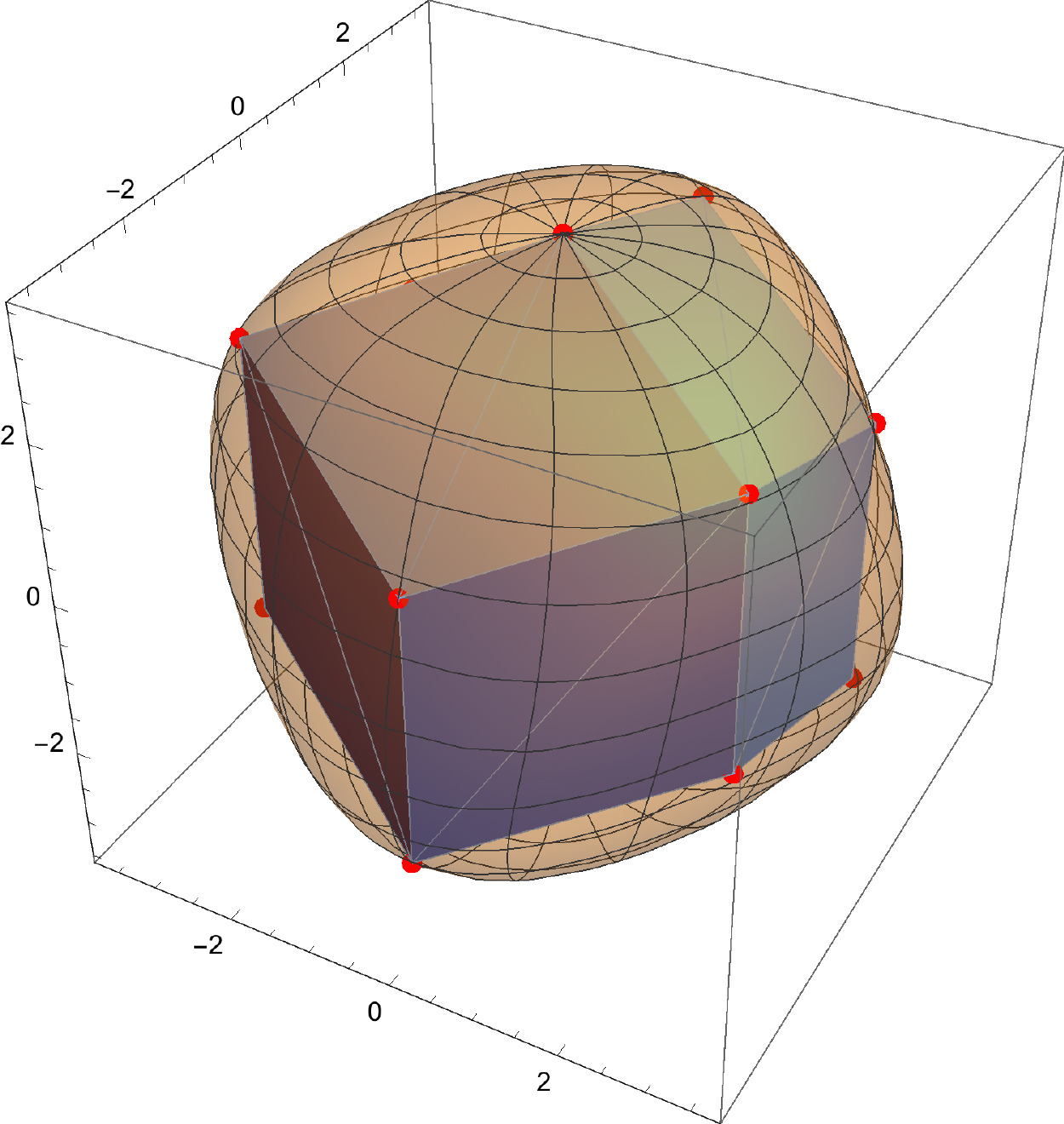}\\
    (a) The phase-locked region and inscribed rhombic
      dodecahedron $\idb4$
 \end{minipage}
  \begin{minipage}[b]{0.45\textwidth}
\centering
    \includegraphics[width=\textwidth]{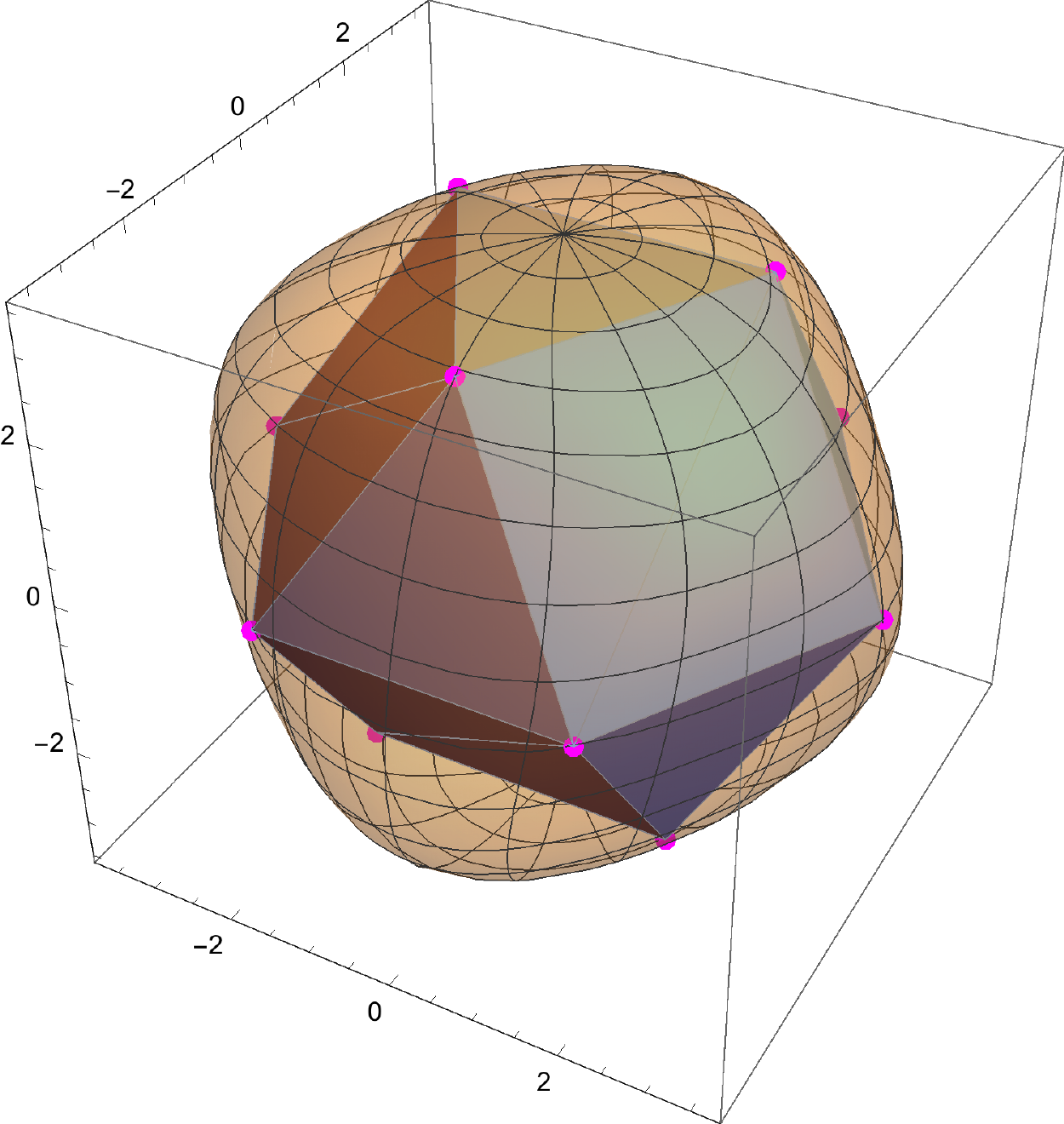}\\
    (b) The phase-locked region and inscribed cuboctahedron $\ics4$
\end{minipage}
\caption{The phase-locked region for $N=4$ and the inscribed polytopes  $\idb4$ and $\ics4$.  The defining points $\rdb4$ and $\rcs4$  are marked by colored vertices.}
\label{fig:inscribed4}
\end{figure}

\begin{figure}[!tbp]
  \centering
  \begin{minipage}[b]{0.45\textwidth}
\centering
    \includegraphics[width=\textwidth]{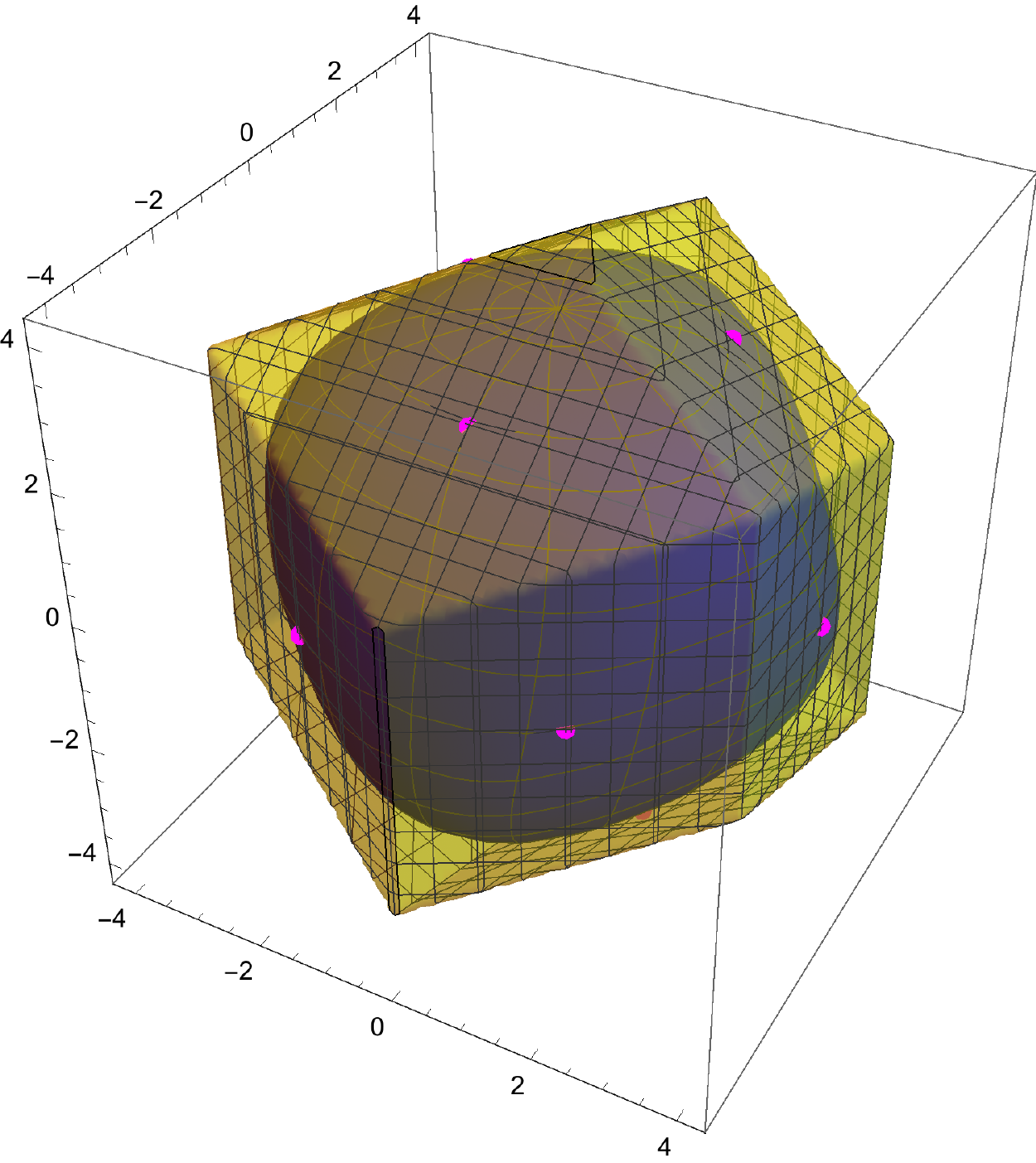}
    (a) The phase-locked region and circumscribed rhombic
      dodecahedron  $\cdb4$
 \end{minipage}
  \hfill
  \begin{minipage}[b]{0.45\textwidth}
\centering
    \includegraphics[width=\textwidth]{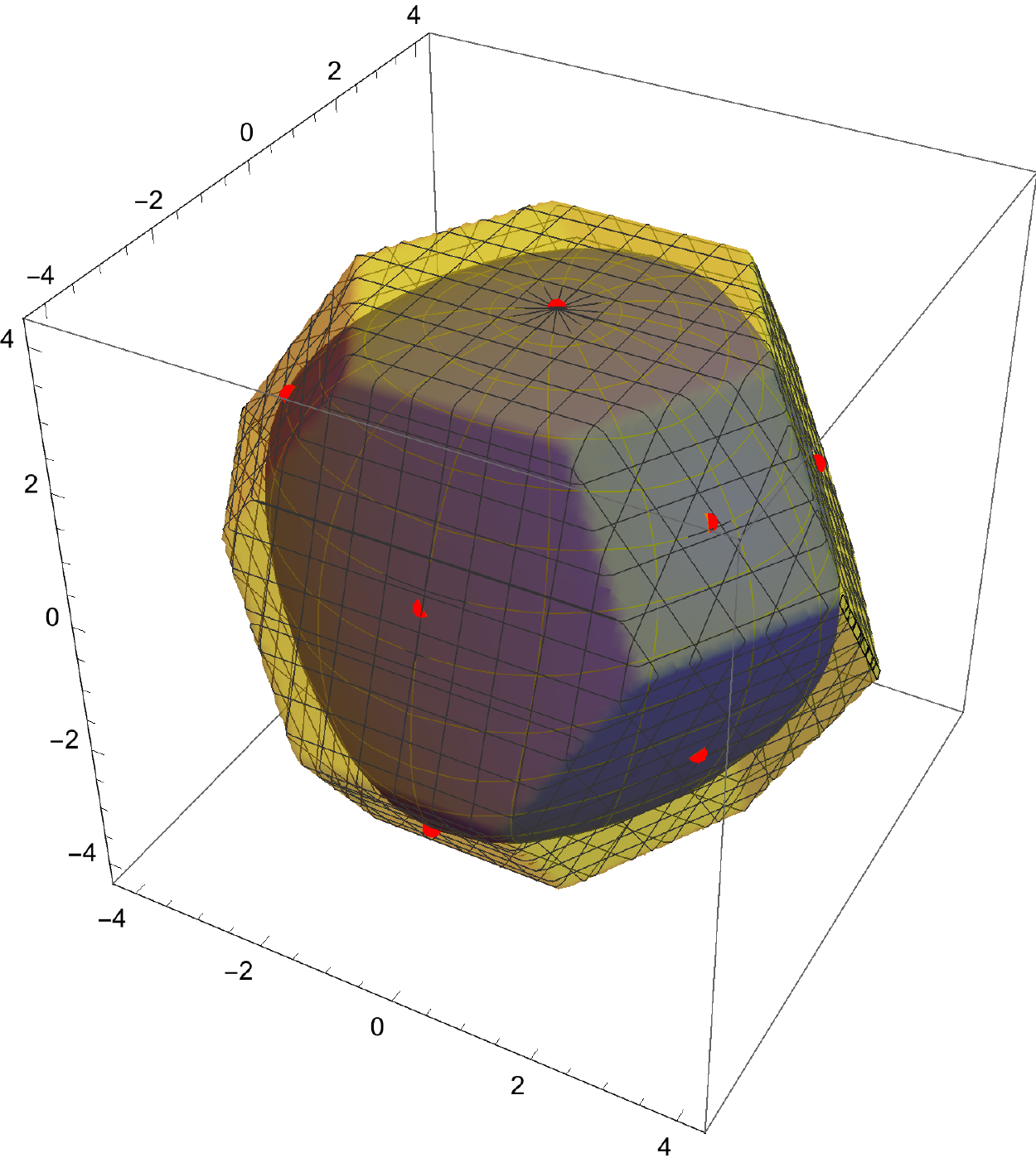}
    (b) The phase-locked region and circumscribed  truncated octahedron   $\ccs4$
 \end{minipage}
\caption{The phase-locked region for $N=4$ and the circumscribing polytopes  $\cdb4$ and $\ccs4$.  The defining points $\rdb4$ and $\rcs4$  are marked by colored vertices.}
\label{fig:circumscribed4}
\end{figure}

\begin{figure}[!tbp]
  \centering
  \begin{minipage}[b]{0.45\textwidth}
    \includegraphics[width=\textwidth]{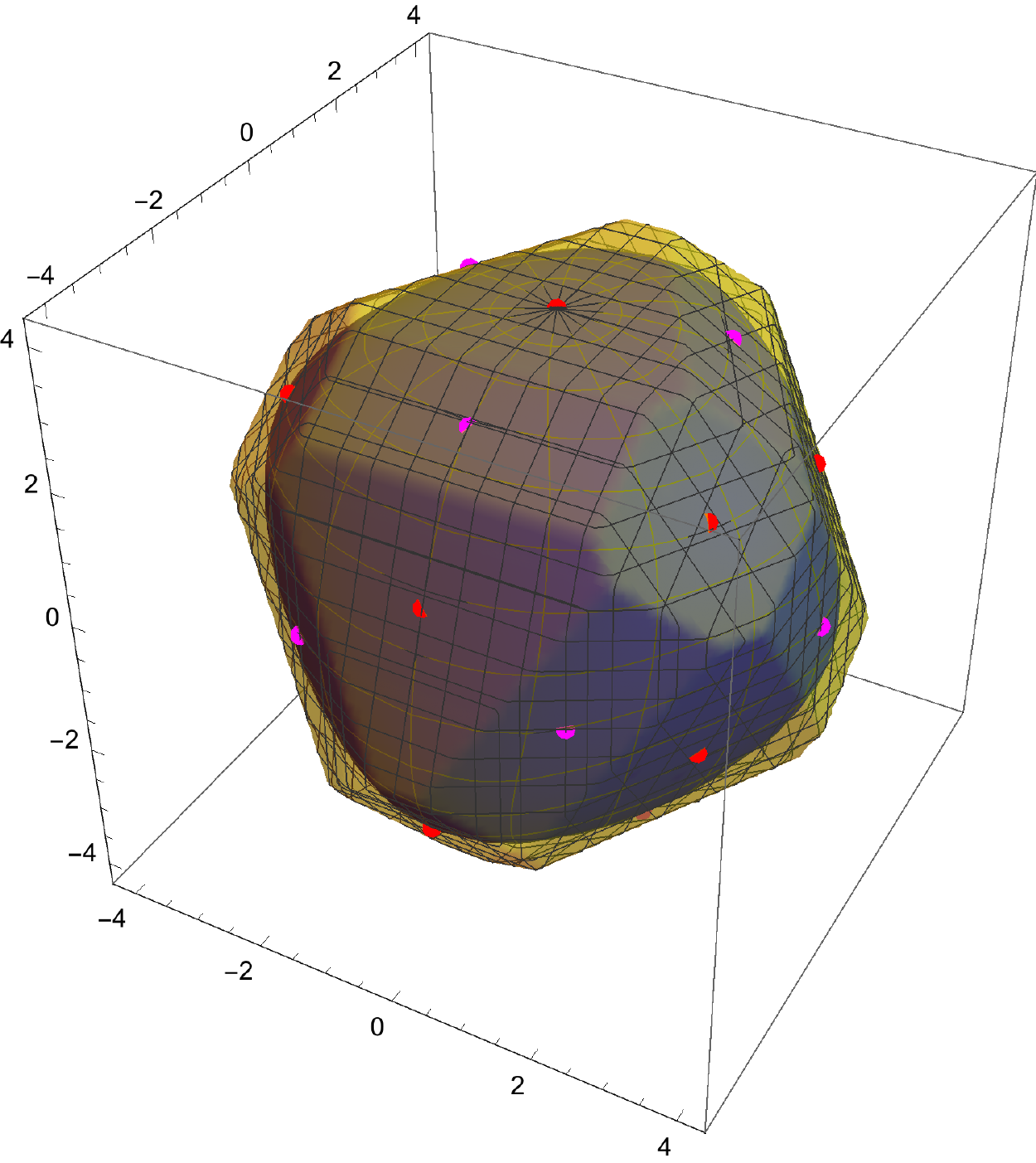}
    \caption{The phase-locked region and circumscribed polytope
      $\cdb4 \cap \ccs4$ }
\label{fig:Intersect}
 \end{minipage}
  \hfill
  \begin{minipage}[b]{0.45\textwidth}
    \includegraphics[width=\textwidth]{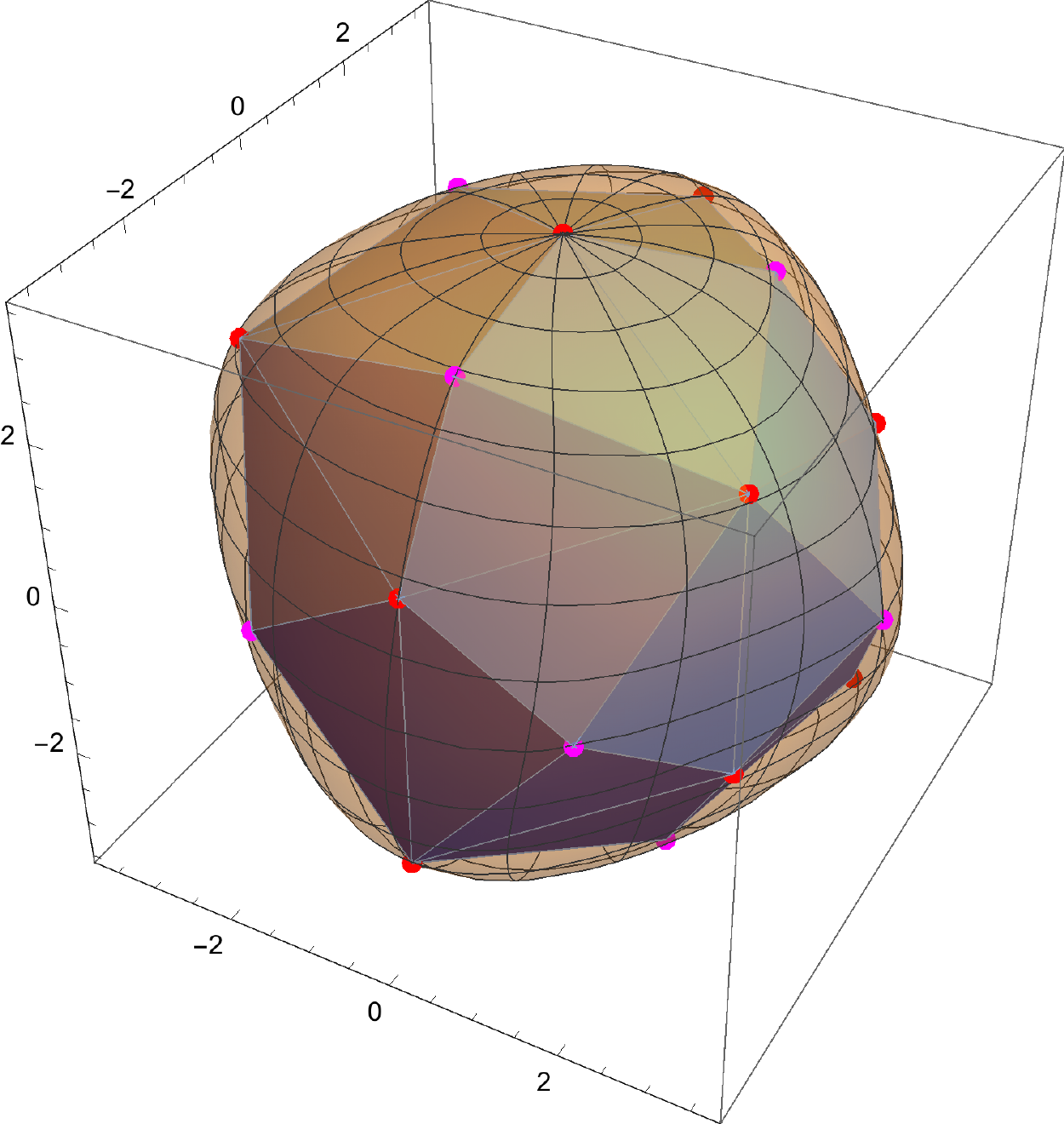}
    \caption{The phase-locked region and inscribed polytope
      $\inscr(\rcs4\cup \rdb4)$ }
\label{fig:Infimal}
\end{minipage}
\end{figure}

Finally we conclude this section by giving the volumes of these
polyhedra as a function of $N$. This will be useful since it gives
some sense of which conditions are in some sense the best --- the best
necessary condition ($\circscr$) is the one with the smallest volume, while
the best sufficient condition ($\inscr$) is the one with the largest
volume.

\begin{prop}\label{prop:volume}
If $\vol_{N-1}(P)$ denotes the $(N-1)$ dimensional Lebesgue volume of a polytope $P \subseteq \R^N_0$ then
the polytopes considered here have the following volumes.
\begin{align*}
&\vol_{N-1}\left(\idb N\right) = N^{N-\frac12} &\\
&\vol_{N-1}\left(\cdb N\right) = 2^{N-1}N^{N-\frac32} &\\
&\vol_{N-1}\left(\ccs N\right) = N^{N-\frac12} \left(\frac{2
  \tau_N}{N}\right)^{N-1}\approx e^{-2} 2^{N-1} N^{N-\frac12}(1+O(N^{-1}))&\\
&\vol_{N-1}\left(\ics N\right) = \frac{\sqrt{N}
  (2(N-1))!}{((N-1)!)^3} (\tau_N)^{N-1}\approx \frac{2^{2N-\frac32}}{2
                                         \pi \sqrt{N}} e^{N-2}(1 + O(N^{-1}))&
\end{align*}
In particular, we have asymptotic bounds for the volume of $\mc D_N$:
\[
\vol_{N-1}\left(\ics N\right)  \le  \vol_{N-1}\left(\idb N\right) \leq \vol_{N-1}\left(\mc D_N\right) \leq  \vol_{N-1}\left(\cdb N\right) \leq \vol_{N-1}\left(\ccs N\right).
\]
Numerically we have found that this order seems to hold for all $N \geq 4$.
\end{prop}

 \begin{remark}
All of these except the last are computed in Conway and Sloane, as they are up to scaling
the volumes of the Voronoi cells of $A_{N-1}$ and $A_{N-1}^*$. We
compute the volume of the last in an appendix, using the
combinatorial results of Postnikov \cite{Postnikov.2009}.
\end{remark}

Note that for large $N$ the polytope
$\ics N$ has substantially smaller volume than the polytope
$\idb N$. This makes a certain amount of intuitive sense: one
expects that having a larger frequency difference leads to the loss of
phase-locking, so one expects that norm should behave like an
$L_\infty$ norm. While the polytope  $\idb N$ is the unit ball of
a norm that is closely related to the $L_\infty$ norm, $\ics N$ is the unit ball of a norm related to the $L_1$
norm. Thus it is perhaps not surprising that the estimate it gives is
quite conservative: For a ``random'' vector in ${\mathbb R}^n$ the
$L_\infty$ norm is smaller than the $L_1$ norm by a factor of roughly
$N$.
Similarly here we see that the volume of the $L_1$-like ball is smaller
than the volume of the $L_\infty$-like ball by a factor of $N^{N-1}$.
 We do, however, stress again that neither polytope
is completely contained within the other.

\section{Natural Norms and Merging Polytopes}

In the last section given a collection $R$ of points on the boundary
of the phase-locked region we defined two polytopes $\inscr(R)$ and $\circscr(R)$ which
were contained in and contained the stable region respectively. Given
a combinatorial description of a polytope it is not always easy to
decide if a given point is contained in the polytope, so our
goal in this section is to express these polytopes as the unit balls
of certain norms. From these representations it will be relatively
straightforward to check if any given frequency vector lies in the
polytope.

We will also show how to ``add'' collections of points: given two sets
of points $R$ and $R'$ on the boundary of the phase-locked region
we show how to relate $\inscr({R \cup R'})$ to $\inscr(R)$ and $\inscr({R'})$, and
analogously for $\circscr(R \cup R')$.

\begin{lem}\label{lem:cr}
Let $\circscr(R)$ be a polytope containing the origin defined as the intersection of a collection of
half-spaces $H_x$ with normal vectors derived from the points $x \in R$. If the set $R$ has
the property that $x \in R \iff -x \in R$ then $\circscr(R)$ is the unit ball
of some (semi-)norm $\circnorm\cdot R$; more specifically,
\[
\circscr(R) = \{ y \in {\mathbb R}^n  \vert \circnorm y R \leq  1 \},
\]
where
\begin{equation*}
  \circnorm y R := \max_{x \in R} \frac{\langle y,x
  \rangle}{\langle x,x \rangle},
\end{equation*}
and where $\ip \cdot\cdot$ denotes the standard Euclidean inner product.
If the number of linearly independent vectors in $R$ is at least $N$
then the semi-norm is actually a norm.
\label{lem:Normals}
\end{lem}

\begin{proof}
The half-spaces containing the polytope $\circscr(R)$ are defined by
\[
\frac{\langle y,x
  \rangle}{\langle x,x \rangle} \leq 1,
\]
so the set of all $y$ such that $\max_{x\in R} \frac{\langle y,x
  \rangle}{\langle x,x \rangle} \leq 1$ is clearly equivalent to the
polytope $\circscr(R)$.  It remains only to check that this defines a
norm.  First note that $\circnorm 0 R=0$ by definition.  Now let $\alpha>0$.  It is clear that
\begin{equation*}
  \arg\max_{x\in \R} \frac{\ip {\alpha y} x}{\ip x x} =   \arg\max_{x\in \R} \frac{\alpha\ip y x}{\ip x x}
\end{equation*}
by scaling, and thus $\circnorm{\alpha y}R = \alpha\circnorm{y}R$.  Now, if $\alpha<0$, then note that
\begin{equation*}
  \frac{\ip{\alpha y}{-x}}{\ip{-x}{-x}} = \frac{\ip{(-\alpha)y}{x}}{\ip x x},
\end{equation*}
(and $-x\in R\iff x\in R$ by assumption) and thus $\circnorm{\alpha y}R = \alpha\circnorm{y}R$.  Putting these two together gives
\begin{equation*}
  \circnorm{\alpha y}R = \av{\alpha}\circnorm{y}R,
\end{equation*}
 so the homogeneity property holds.

 We also compute
 \begin{align*}
&\circnorm{y+y'}R=\max_{x \in R} \frac{\langle y+y',x
  \rangle}{\langle x,x \rangle} \\&=  \max_{x \in R} \left(\frac{\langle \alpha y,x
  \rangle}{\langle x,x \rangle} + \frac{\langle \alpha y,x
  \rangle}{\langle x,x \rangle} \right) \leq \max_{x \in R} \frac{\langle \alpha y,x
  \rangle}{\langle x,x \rangle} +\max_{x \in R} \frac{\langle \alpha y',x
  \rangle}{\langle x,x \rangle} = \circnorm{y}R + \circnorm{y'}R,
\end{align*}
so the triangle inequality holds. If there are at least $N$
independent vectors in $R$ then $R$ contains a basis so if $y$ is
non-zero there is at least one element of $R$ with a non-zero
projection on $y$, and hence at least one element with a positive
projection on $y$.
\end{proof}

\begin{lem}Given two collections of boundary points we have that
\begin{equation*}
  \circscr(R\cup R') = \circscr (R) \cap \circscr(R')
\end{equation*}
and
\begin{equation*}
  \circnorm{\cdot}{R\cup R'} = \max\left(\circnorm\cdot R,\circnorm \cdot {R'}\right).
\end{equation*}\end{lem}
\begin{proof}The first statement follows from the definition of $\circscr(R)$ as an intersection of half-planes.  The second is a corollary of Lemma~\ref{lem:cr}:
\begin{equation*}
  \circnorm y {R\cup R'} = \max_{x\in R\cup R'} \frac{\langle y,x
  \rangle}{\langle x,x \rangle} = \max\left(\max_{x\in R} \frac{\langle y,x
  \rangle}{\langle x,x \rangle},\max_{x\in R'} \frac{\langle y,x
  \rangle}{\langle x,x \rangle}\right) = \max\left(\circnorm y R,\circnorm y {R'}\right).
\end{equation*}
\end{proof}

Next we consider the case of the $\inscr$ polytopes. In general, the characterization
of the norm in terms of the vertices does not seem to be as nice as the
characterization of the  norm in terms of the normals to the
supporting half-spaces, but for very special polytopes (permutahedra)
a classical result of Rado gives a characterization:

\begin{thm}[Rado\cite{Rado.1952}]
Consider the permutahedron given by all convex combinations of permutations
of a vector $v$. We can assume that the coordinates of $v$ are ordered
$v_1\geq v_2 \geq \ldots v_{n-1} \geq v_n$. Given an arbitrary vector $x$ the
vector is in the permutahedron if and only if all permutations of $x$ are in
the permutahedron, so we can assume without loss of generality that
$x$ is ordered the same way. Then  $x$ is in the permutahedron
if and only if the inequalities
\[
\sum_{j=1}^k x_j \leq \sum_{j=1}^k v_j \quad \forall k\in 1\ldots n-1
\]
and the equality
\[
\sum_{j=1}^n x_j = \sum_{j=1}^n v_j.
\]
hold.
\end{thm}
The polyhedra formed by $\rcs{N,j}$ are all permutahedra and this theorem will enable us to define a norm whose unit ball is $\ics{N,j}$.  As $\idb N$ is not a permutahedron, we will need a slightly different approach to determine the associated norm.  To do so, we will need a more general result.

For a general set $R$, once we have constructed $\inscr(R)$ it remains to show how to combine collections of points --- in other words how to relate $\inscr({R \cup R'})$ to $\inscr(R)$ and $\inscr({R'})$. It is clear from the definition of $\inscr(R)$ as the convex hull of the points in $R$ that we have $\inscr(R\cup R') = \conv(\inscr(R),\inscr(R'))$, where $\conv(A,B)$ denotes the convex hull of $A \cup B$. On the level of the norms this can be expressed as follows:

\begin{prop}
Let $R$ and $R'$ be any collection of points such that the convex hulls $\inscr(R)$ and $\inscr(R')$ are balanced and absorbing.  {(Recall Definition~\ref{def:balanced-absorbing}.)}  Then the norm corresponding to $\inscr({R\cup R'})$ is given by the
infimal convolution of the norms corresponding to $\inscr(R)$ and $\inscr({R'})$, i.e. if we define:
\begin{equation}\label{eq:defofRR'}
  \inscnorm{y}{R\cup R'} := \inf _x \left(\inscnorm x R + \inscnorm{y-x}{R'}\right)
\end{equation}
then $\inscr(R\cup R')$ is the unit ball under $\inscnorm\cdot{R\cup R'}$.
\end{prop}

\begin{proof}
Recall from Theorem~\ref{thm:simon} that  since  $\inscr(R)$ and $\inscr(R')$ are balanced and absorbing, they
each have the property that they are the closed unit ball of a seminorm, i.e. that there exist $\inscnorm\cdot R$ and $\inscnorm\cdot{R'}$ such that
\begin{equation}
  \inscr(R) = \left\{y\in \R^n : \inscnorm y R \le 1\right\},\quad   \inscr(R') = \left\{y\in \R^n : \inscnorm y {R'} \le 1\right\}.
\end{equation}

Let $v\in \inscr(R), v'\in \inscr(R')$, and assume that $y$ is a convex combination $y=\alpha v + (1-\alpha) v'$, $\alpha\in[0,1]$.   If we plug $x = \alpha v$ into the infimum we see that
\begin{equation*}
  \inscnorm{y}{R\cup R'} = \inf _x \left(\inscnorm x R + \inscnorm{y-x}{R'}\right) \le \alpha\inscnorm v R + (1-\alpha)\inscnorm v {R'} \le 1.
\end{equation*}
Therefore $\inscr({R\cup R'})$ is contained in the unit ball under $\inscnorm{\cdot}{R\cup R'}$.

Conversely, if $\inscnorm{y}{R\cup R'}\le 1$ then  by definition there exists $x$ such that $\inscnorm x R + \inscnorm {y-x}{R'} \le 1$.   If $x=0$ then $y\in \inscr(R') \subset \inscr(R\cup R')$; similarly, if $x=y$ then $y\in\inscr(R)\subset \inscr(R\cup R')$.  If neither ${x}$ nor $y-x$ is the zero vector we have $$y = \inscnorm x R \frac{x}{  \inscnorm x R} + \inscnorm {y-x} {R'} \frac{y-x}{ \inscnorm {y-x} {R'}}.$$    Thus we have written  $y$ as a convex combination of vectors in $\inscr(R)$ and $\inscr(R')$ and we are done.

\end{proof}

\begin{define}\label{def:spread}
For a vector $y = (y_1,\dots, y_n) \in \R^n$, we define $\ymax, \ymin, \ymaxx l, \yminn l$ as follows:  let $z$ be the vector in $\R^n$ with the entries of $y$ sorted in an increasing fashion (i.e. $z\in \Sym(y)$ and $z_i \le z_{i+1}$), and then define
\begin{equation*}
  \yminn l = z_l,\quad \ymaxx l = z_{n-l}.
\end{equation*}
That is to say, $\yminn l$ is the ``$l$th smallest'' entry of $y$ and $\ymaxx l$ is the ``$l$th largest'' entry of $y$.  We then define $\ymax := \ymaxx 1$ and $\ymin := \yminn1$.
Let us also define the {\bf spread seminorm} of a vector $y$ as $\snorm{y} := \ymax-\ymin$.
We also use the standard notation $\norm y_1 = \sum_{i=1}^N \av{y_i}$ and $\norm y_\infty = \max_{i=1}^N \av{y_i}$.
\end{define}

\begin{prop}\label{prop:norms}
The polytopes $\idb N$, $\ics N$, $\ics{N,j}$, $\cdb N$, $\ccs N$, $\ccs {N,j}$
can be defined in terms of the following norms:
\begin{align*}
\idb N &= \{ y\in\R^N_0 \colon \snorm y \leq N \}&  \\
\ics N &= \{ y\in\R^N_0 \colon  {\norm y}_1   \leq 2\tau_N \} & \\
\ics {N,j} &=\left \{ y\in\R^N_0 \colon  \max\left(j^{-1} \norm y_1 ,  2\Vert y\Vert_\infty\right)  \leq 2\tau_{N,j} \right\} &
                                                                  \\
\cdb N &= \left\{ y\in\R^N_0 \colon  \max_{k \in \{1\ldots N-1\}} \max_{|S|=k}
                \frac{\sum_{l \in S}y_l}{k(N-k)}\leq 1\right\} &\\
\ccs N &= \{ y\in\R^N_0 \colon \snorm y \leq 2 \tau_N\}
                                                                & \\
\ccs{N,j} &= \left\{ y\in\R^N_0 \colon  \sum_{l=1}^j y_{\max,l} -
  \sum_{l=1}^j y_{\min,l} \leq 2 j  \tau_{N,j}\right\} &
\end{align*}
In the last case the norm is given by the sum of the largest $j$
elements minus the sum of the smallest $j$ elements. This is only
defined for $2j\leq N.$ One can also define a norm whose unit ball is the intersection of {\bf ALL} the generalized Chopra-Spong conditions:
\[
\tilde {\ccs N} = \bigcap_{j=1}^{\lfloor \frac{N}{2} \rfloor} \ccs{N,j} =\left\{ y: \max_{k \in \{1\ldots \lfloor\frac{N}{2}\rfloor\}}\max_{|S_1|=k,|S_2|=k} \frac{\sum_{i\in S_1}y_i - \sum_{j\in S_2}y_j}{2 k \tau_{N,k}} \leq 1\right\} .
\]
Recalling that $N-2j \leq \tau_{N,j} \leq N-j$ this formula looks very similar to the one defining $\cdb N$.
\end{prop}

\begin{remark}
It is worth remarking that the polytopes $\ics N$, $\idb N$, $\ccs N$, and $\cdb N$ are  connected to the $A_{N-1}$ root lattices
and the dual lattices $A_{N-1}^*$. The polytopes $\idb N$ and $\ccs N$ are
(up to scaling) the Voronoi cells of the $A_{N-1}$ lattice. $\ics N$ is (again up to scaling) the unit ball of the dual norm to
the norm defining $\idb N$ and similarly for $\ccs N$.  It is easy to
check that the dual norm to the norm $y_{\text{max}}-y_{\text{min}}$ (in the space of mean zero vectors!) is one
half the standard $L_1$ norm (again in the space of mean zero
vectors). Finally  $\cdb N$  is the Voronoi cell of the dual
lattice $A_{N-1}^*$. We refer the interested reader to the text of
Conway and Sloane for details~\cite{CS}.

\end{remark}

Having derived these norm conditions, we can proceed to combine them
as outlined earlier in the section.

\begin{example}[An improved necessary condition.]
As previously discussed the $\ccs N$ polytope leads to the
necessary condition for synchronization
\[
y_{\max}-y_{\min} \leq 2 \tau_N,
\]
as originally derived by Chopra--Spong. Analogously the  $\cdb N$
polytope leads to a dual necessary condition for synchronization
\[
\max_{k \in \{1\ldots N-1\}} \max_{|S|=k}
                \frac{\sum_{l \in S}y_l}{k(N-k)}\leq 1.
\]
As discussed in the earlier example for $N=4$ these conditions reduce
to a rhombic dodecahedron of volume $256$ and a truncated octahedron
of volume $162 \sqrt{3} \approx 280.6$. One can trivially combine these
 conditions and obtain the improved necessary condition being that
{\bf both} of these conditions must hold. This gives a new, smaller
polyhedron containing the phase-locked region. Since this region is
defined by a collection of linear inequalities it is elementary,
although tedious, to compute the volume. A symbolic computation using
Mathematica gives the volume of the intersection of these figures as $(126 \sqrt{3} -210) \sqrt{7 - 4
  \sqrt{3}} + 3642 \sqrt{3} - 6074\approx 236.34, $ as compared with a
volume of approximately $210$ for the exact phase-locking region. The
resulting polytope is illustrated in Figure \ref{fig:Intersect}, and
takes the form of an octahedron whose edges have been chamfered and
whose vertices have been truncated. The resulting figure has 26 faces
(14 normals from $R^{\text{DB}}_3$ and 12 normals from $R^{\text{CS}}_3$: 12
rectangular faces from chamfering the edges, 8 hexagonal faces coming
from the original faces of the octahedron,and 6 octagonal faces from
truncating the vertices.  Similarly we also give the improved
sufficient condition for phase-locking
\[
\inf_z \left(\frac{z_{max}-z_{min}}{N} + \frac{\Vert y-z \Vert_1}{2 \tau_N}
\right)\leq 1.
\]
The polytope satisfying these conditions is shown in Figure
\ref{fig:Infimal}. It results from applying Conway's {\it kis}
operation to the rhombic dodecahedron -- raising a pyramid on each
rhombic face. We have not computed the volume analytically but
numerical integration gives the volume as $166.28$. Compare this with
$128$ for the rhombic dodecahedron and $\frac{135\sqrt{3}}{2}\approx
116.913$ for the cuboctahedron.

\end{example}

\section{Numerical simulations}

\subsection{Our method}

In this section,  we present some numerical results using Monte Carlo simulations on the relative sizes of the various inscribed and circumscribed regions defined in the previous sections of the paper.   For us to be able to do this, some of the elements of the Monte Carlo simulations had to be   specifically tailored to the problem at hand. We believe that this method is likely to be of independent interest, and so we present it in some detail.

Let us stress that one cannot expect to just use any ``naive'' method to sample any of our polytopes and get a reasonable result.  For example, we might think that we could just sample from a circumscribing hypercube or hypersphere and then use accept/reject (since we have explicit accept/reject criteria in Proposition~\ref{prop:norms}).  However, we are guaranteed to run into a ``curse of dimensionality'' for even moderate $N$ (q.v.~the difference of the volume bounds in Proposition~\ref{prop:volume} and those of circumscribing spheres or cubes).  As such, it is required that we find a method to efficiently sample at least one of the polytopes directly before we can make progress.

One of the main elements in the method below is the fact that the polytope $\ccs N$ is the image of a hypercube  under a projection map, so a uniform sample of the hypercube projects to a weighted sample of $\ccs N$ with known weighting.  Since we can sample the hypercube, we can then design a method to sample $\ccs N$. This is the basic idea, details below.

\begin{definition}\label{def:Q}
 We define $\proj$ as the orthogonal projection from $\R^N$ to the mean-zero subspace ${\mathbb R}_0^N$.  We denote by $\mc{Q}_N$ the standard (filled) hypercube $[-1,1]^N$, and then $\tncn$ is the (filled) hypercube $[-\tau_N,\tau_N]^N$.
\end{definition}

\begin{prop}[Pok\'{e} Method]
If $f\colon{\mathbb R}^N \rightarrow {\mathbb R}$ is any bounded function satisfying $f(\proj x) = f(x)$ (i.e. $f(x)$ is independent of the mean of  $x$),
then
\begin{equation}\label{eq:MCint}
\int_{\tncn} \frac{f(x)}{\sqrt{N} (2\tau_N - \snorm x)} dx = \int_{\omega \in \ccs N } f(\omega)\, d\omega,
\end{equation}
where $d\omega$ is the usual $(N-1)$-dimensional Lebesgue measure on ${\mathbb R}_0^N\cong {\mathbb R}^{N-1}$.  This allows for an explicit Monte Carlo sampler for any function supported on $\ics N$ as follows: if we sample the hypercube $\tncn$ a total of $\msamp$ times, then as $\msamp\to\infty$, we have
\begin{equation}\label{eq:MCsum}
  \frac 1 \msamp  \sum_{k=1}^\msamp \left[ \frac{(2\tau_N)^{N}}{\sqrt{N}}\times\frac{f(X_i)}{\left(2\tau_N - \snorm {X_i} \right)}\right] \xrightarrow[\msamp \to \infty]{} \int_{\omega \in  \ccs N} f(\omega) \, d\omega.
\end{equation}
 in the usual law of large numbers sense (in particular, this convergence is valid  almost surely).
 \label{prop:Sampling}
\end{prop}

\begin{remark}
  Note that we can sample the left-hand side more or less explicitly.  The coordinates of $X_i$ are independent, so we can just choose $X_i^{(k)} \sim U[-\tau_N,\tau_N]$ and concatenate them to obtain $X_i$.  Moreover, given $X_i$ we can evaluate the summand inside of square brackets explicity; from this we just repeat $M$ times and take the mean.

We stress that this gives us a method to sample any region that is a subset of $\ccs N$. This includes all of the inscribed polyhedra and any of the circumscribed polyhedra that include $\rcs N$.

Also note that we could use this method to measure whatever weighted volume that we would like on the $\ccs N$ polytope if we wanted, although all we consider below are indicator functions of various other polytopes.

\end{remark}

\begin{proof}

The proof follows from a rotation and a partial integration. We first parameterize the hypercube $\tncn$ by its fiber representation over $\ics N$: specifically, for any $x\in \tncn$, we write
\begin{equation*}
  x = \omega +\frac{\sigma}{\sqrt{N}} (1,1,1,\ldots,1)^\intercal,  \mbox{ where $\omega = \proj x $ and $\sigma = \sum_i x_i$}.
\end{equation*}
For each $\omega\in\ccs N$, denote by $I_\omega$  the set of all $\sigma$ such that $\omega + \sigma/\sqrt{N} (1,1,1,\ldots,1)^\intercal  \in \tncn$.  Note that $I_\omega$ is a subinterval of the real line that is symmetric around zero.  Recall that $f(x) = f(\proj x=f(\omega)$, and note that $x\mapsto (\omega,\sigma)$ is an orthogonal transformation, and thus
\begin{equation*}
  \int_{\tncn} f(x) \,dx = \int_{I_\omega} \int_{\ccs N}   f(\omega)\,d\omega \,d\sigma  = \av{I_\omega} \int_{\ccs N}   f(\omega)\,d\omega.
\end{equation*}
Now it remains to compute $\av{I_\omega}$.
To see this note that membership in the cube is typically defined by the $2N$ inequalities
\begin{equation*}
 \omega_i + \frac{\sigma}{\sqrt{N}}  < \tau_N \qquad \forall i,\quad
 \omega_i + \frac{\sigma}{\sqrt{N}}  > -\tau_N \qquad \forall i,
\end{equation*}
but we can convert this to the necessary and sufficient conditions
\begin{equation*}
 \max_i \omega_i + \frac{\sigma}{\sqrt{N}}  < \tau_N,\quad
 \min_i \omega_i + \frac{\sigma}{\sqrt{N}}  > -\tau_N,
\end{equation*}
or equivalently
\[
-\tau_N - \min_i \omega_i    \leq \frac{\sigma}{\sqrt{N}} \leq \tau_N -  \max_i \omega_i
\]
and thus
\begin{equation*}
  \av{I_\omega}  =\left.\begin{cases} \sqrt{N}(2 \tau_N -\snorm\omega), &  2 \tau_N -\snorm\omega\ge 0,\\ 0, & 2 \tau_N -\snorm\omega <0.\end{cases}\right\}
\end{equation*}
Therefore (noting that the support of $\proj(\tncn)$ is exactly  $\ccs N$),
\begin{equation}\label{eq:intf}
\int_{\tncn} f(x) dx = \int_{\omega \in \ccs N } \sqrt N (2 \tau_N-\snorm\omega) f(\omega) d\omega
\end{equation}
for any function $f$ such that $f(\proj x) = f(x)$.  If we note that $\snorm x = \snorm {\proj x}$ and write
\begin{equation*}
   g(x) = \frac{f(x)}{\sqrt N (2 \tau_N-\snorm x)},
\end{equation*}
it follows that $g(\cdot)$ also has the property that $g(\proj x) = g(x)$.  Therefore, reusing~\eqref{eq:intf} with $f$ replaced by $g$ gives
\begin{equation*}
  \int_{\tncn} g(x) dx = \int_{\omega \in \ccs N } \sqrt N (2 \tau_N-\snorm\omega) g(\omega) d\omega,
\end{equation*}
or
\begin{equation*}
  \int_{\tncn}\frac{f(x)}{\sqrt N (2 \tau_N-\snorm x)}\,dx = \int_{\omega\in \ccs N} f(\omega)\,d\omega,
\end{equation*}
which is exactly~\ref{eq:MCint}.

Now we might think that~\eqref{eq:MCsum} follows directly from~\eqref{eq:MCint} (and it normally would) but we have to be a bit careful:  we need to show that the summand has finite mean to get the standard LLN convergence.    Here $N$ is fixed, so all of the prefactors involving $N$ won't matter, and we assumed above that $f$ is bounded.  The only challenge that remains is that we need to show that if $X$ is uniform in the hypercube $\tncn$, then
\begin{equation}\label{eq:finitemean}
  \E\left[\frac{1}{2\tau_N - \snorm X}\right] < \infty.
\end{equation}
Note that this random variable is essentially unbounded (it blows up at the boundary of $\ics N$) so we need to be careful.   So, some notation.  Let $Y^{(k)}$ be independent $U(0,1)$ random variables, and let $X^{(k)} = \tau_N(2Y^{(k)}-1)$, then $X^{(k)}$ are independent $U(-\tau_N,\tau_N)$ and the vector $X$ is a sample of $\tncn$.  Note that $\snorm{X} = 2\tau_N\snorm{Y}$ exactly.  Then we have
\begin{equation*}
  \P\left(\frac{1}{2\tau_N - \snorm{X}} > K\right) = \P\left(\snorm{Y} > 1-\frac{1}{2K\tau_N}\right)
\end{equation*}
by some basic algebra.

Finally note that $\snorm{Y}$ is the ``sample range'' of $N$ {independent and identically-distributed} $U(0,1)$, and it is well known~\cite[Chapter 3]{David2003OrderStats.book} that the distribution for $\snorm{Y}$ is $\mathrm{Beta}(N-1,2)$, and in particular
\begin{align*}
  \P\left(\snorm{Y} > 1-\frac{1}{2K\tau_N}\right) &= \int_{1-1/(2K\tau_N)}^1 N(N-1)x^{N-2}(1-x)\,dx \\
  &= 1-\frac{N(1-1/(2K\tau_N))^{N-1}}{2K\tau_N} - \left(1-\frac{1}{2K\tau_N}\right)^N\\
  &= \frac{N^2-N}{4\tau_N^2} K^{-2} + O(K^{-3}).
\end{align*}
In particular we have that there is a $C = C(N)$ such that
\begin{equation*}
  \P\left(\frac{1}{2\tau_N - \snorm{X}} > K\right) \le C(N) K^{-2}
\end{equation*}
and using Markov's Inequality this implies~\eqref{eq:finitemean}. Note that since we only decay at rate $K^{-2}$, we don't expect that this random variable has finite variance, and so the Monte Carlo method might converge slowly in $\msamp$.
\end{proof}

\subsection{Numerical Results: Circumscribed Polyhedra}

We will consider the case of circumscribed polyhedra first, as it is somewhat simpler to implement numerically. To begin with we note that all of the norms defined in Proposition~\ref{prop:norms} can be computed efficiently and are thus valid accept/reject observables. Extraction of the largest or smallest element of a list case be done in time $O(N)$, so the $\ell_\infty$ norm and spread norm $\snorm\omega$ can be computed in time $O(N)$, where $N$ is the dimension, as can the $\ell_1$ norm. The norm defining the circumscribed D\"orfler-Bullo region involves maximizing over subsets of different sizes
\[
\max_{k\in 1\ldots N} \max_{|S|=k} \frac{\sum_{l \in S}y_l}{k(N-k)}.
\]
Despite this combinatorial description this quantity can be computed in time $O(N \log(N))$. To see this note that of all subsets of cardinality $k$ it suffices to consider only the subset containing the $k$ largest elements. If one sorts the entries of $y$ ($N\log(N)$ via mergesort or similar) and then constructs the vector of partial sums of the sorted $y$ (time $O(N)$) then the functional above is the (weighted) maximum entry of the vector of partial sums, so this quantity can be computed in time $O(N\log(N))$.

We have performed some numerical experiments to compute the volumes of
the intersections of the various circumscribed polytopes in different
dimensions. The volumes of the polytopes $\ccs N$ and
$\cdb N$ were computed analytically using the formulae
derived earlier in the paper, while the volumes of the remaining
polytopes were computed using Monte-Carlo sampling with $\msamp=10^6$
points using the scheme outlined above. To briefly summarize we generate a sample $X$
of the cube $\tncn$, compute $\omega = \proj X$, the
projection of the point $x$
into the mean-zero subspace ${\mathbb R}_0^N$.  We then compute the
norm(s) of $\omega$ defining membership in the given polytope; the
point is counted with weight $1/(2 \tau_N - \snorm\omega)$ if it belongs to the polytope and is not
counted if it does not belong to the polytope. The results are given
in table \ref{table:cMC}: we give the volumes of the various
polytopes, while the quantity in brackets represents the fraction of
volume of the Chopra-Sprong polytope $\ccs N$, the previously best-known
necessary condition. One can see that $\av{\cdb N}/\av{\ccs N}$ tends to zero algebraically (as
we know rigorously from the analytic formulae and asymptotics) and
that the volumes of the other polytopes are comparatively smaller.

We have also used this sampling algorithm to compute a numerical approximation to the true volume of the stably phase locked region.
We did this using the well-known equation for the order parameter $r$,
\begin{equation}
r = \frac{1}{N} \sum_{i=1}^N \sqrt{1-\frac{\omega_i}{r^2}}.
\label{eqn:OP}
\end{equation}
Note that this equation holds in the mean-field scaling, and must be rescaled for the conventions used in this paper. Existence of a stably phase-locked solution is equivalent to the  existence of a root of  Equation (\ref{eqn:OP}). It follows from Jensen's inequality that $\sqrt{\frac{1}{N} \sum_i (1-\frac{\omega_i}{r^2}}) > \frac{1}{N} \sum_{i=1}^N \sqrt{1-\frac{\omega_i}{r^2}},$ so there can be no roots for $r>1$. The difference $r - \frac{1}{N} \sum_{i=1}^N \sqrt{1-\frac{\omega_i}{r^2}}$ is obviously positive for sufficiently large $r$, so there exists a stably phase-locked solution if this quantity is anywhere negative. The function is only defined for $r>\max_i \omega_i$ so we numerically assess the existence of a stably phase-locked fixed point by sampling the function $r - \frac{1}{N} \sum_{i=1}^N \sqrt{1-\frac{\omega_i}{r^2}}$ at twenty points in the interval $r \in [\max_i \omega_i, 1]$: if the minimum over these samples is negative then there necessarily exists a zero of the function and thus a stably phase-locked fixed point. This is, it should be said, more expensive computationally than assessing membership in the various circumscribed polyhedra but is still computationally tractable.

\begin{table}
\begin{center}
  \resizebox{\columnwidth}{!}{
  \begin{tabular}{ || c | c | c | c | c | c | c | | }
    \hline
        $N$
    & $\av{\ccs N}$
    & $ \av{\cdb N}$
    & $\av{\cap_{j=1}^{\lfloor\frac{N}{2}\rfloor}\ccs {N,j}}$
    & $ \av{\cdb N \cap \ccs N}$
    & $\av{\cdb N  \cap_{j=1}^{\lfloor\frac{N}{2}\rfloor} \ccs {N,j}}$
    & True Volume\\ \hline \hline

    5 & $5277.32$ & $4472.14$ [0.84] &$4660$ [0.88] & $4080$ [0.77]  & $3950$ [0.75] & 3210 [0.61]\\ \hline
    10 & $2.815\times10^{11}$  & $1.619\times10^{11}$ [0.58] &$1.56\times10^{11}$ [0.55]& $1.46\times10^{11}$ [0.52]
         &$1.16\times10^{11}$ [0.41] & $5.26\times10^{10}$ [0.19] \\ \hline

    15 & $2.93\times 10^{20}$ & $1.23\times 10^{20}$ [0.42]&$1.03\times10^{20}$ [0.35]&$1.11\times 10^{20}$ [0.38] & $ 7.01\times 10^{19}$[0.24]
         & $1.66\times 10^{19} $ [0.057] \\ \hline
    20 & $1.86\times 10^{30}$ & $6.15\times 10^{29}$ [0.33]& $4.37\times10^{29}$ [0.23] &$5.6\times 10^{29} $[0.30]
         & $2.74\times 10^{29}$ [0.14] & $3.3\times10^{28}$ [.017]
    \\ \hline

\end{tabular}}
\end{center}
\caption{Volumes of various circumscribing polyhedra for various $N$}
\label{table:cMC}
\end{table}

\begin{table}
\begin{center}
  \begin{tabular}{ || c | c | c | c | c | c | || }
    \hline\hline
    $N$ &  $\msamp$ & True Volume & $\inscr(\rcs N \cup \rdb N)$ & $\idb N$ & $\ics N$ \\ \hline
    5 & $10^6$ & 3210 & 2032 & 1398 & 962 \\ \hline
    10 & $1.5\times10^4$ & $5.26\times10^{10}$ & $6.7 \times 10^9$ & $3.2 \times 10^9$ & $7\times10^7$ \\ \hline
    15 & $500$ & $1.66\times 10^{19}$ & $2\times10^{17}$ & $1\times10^{17}$ & $8\times10^{12}$ \\ \hline

\end{tabular}
\end{center}
\caption{Volumes of various inscribing polyhedra for various $N$}
\label{table:iMC}
\end{table}

\subsection{Numerical Results: Inscribed Polyhedra}

The problem of the inscribed polyhedra is somewhat more difficult to
approach analytically, as it is not obvious how to compute the infimal convolution of two norms as defined in~\eqref{eq:defofRR'}
in a numerically efficient manner. The most obvious approach bypasses
the norms entirely --- if one begins with the vertices of the polytopes, denoted by
$\{{\bf v}_i\}_{i=1}^k$, then we have the problem of deciding if a given
vector ${\bf b}$ lies in the convex hull of these vertices. This can be rather straightforwardly recast as a
linear programming problem. The general linear programming problem is
to
\begin{align*}
  \text{Minimize}: \qquad & {\bf c}^T {\bm \alpha} & \\
  \text{subject to:} \qquad & {\bf A} {\bm \alpha} \leq {\bf b} & \\
  \text{and}& \qquad {\bm \alpha} \geq {\bf 0}. &
\end{align*}
Here the inequalities are understood to hold termwise. In our case we
can take the matrix ${\bf A}$ to be the $M \times k$ matrix having the
vertex vectors  $\{{\bf v}_i\}_{i=1}^k$ as columns, and ${\bf c}$ to
be the vector with all entries $1$. The solution to the linear program
gives the representation of ${\bf b} = \sum_{i=1}^k \alpha_i {\bf v}_i$
having  $\alpha_i\geq 0$ and $\sum_i \alpha_i$ a minimum. Obviously
${\bf b}$ lies in the convex hull of $\{{\bf v}_i\}_{i=1}^k$ if
$\sum_i \alpha_i\leq 1.$ This is, it must be said, a much more numerically challenging problem than the circumscribed problem, as
the number of vertices (and thus the time to solve the linear programming problem) grows exponentially with the dimension.  As such, in this case we could only take fewer samples, so we also report $\msamp$ as a function of $N$ as well.

\section{Phase-locking probabilities and Extreme Value Statistics}

A classical question in the Kuramoto literature is the following question:  If we sample $\omega_i$  from a fixed distribution, what is the probability of the system supporting a phase-locked solution?   Using the notation above, this is equivalent to the question of whether the vector $\omega\in \mathcal{D}_N$.  Of course, obtaining an exact probability is likely to be difficult --- as we have seen above, $\mathcal{D}_N$ has a complicated boundary that is difficult to describe in detail.  However, we can use some of the formulas developed above to obtain bounds of the probabilities, and we show that in certain scalings the  probability of phase-locking demonstrates phase-transition behavior.  Let us write the Kuramoto model as
\begin{equation}\label{eq:gN}
  \frac{d}{dt}\theta_i = \omega_i + \frac{\gamma_N}{N} \sum_{j=1}^N \sin(\theta_j-\theta_i).
\end{equation}

Note that we have changed the equation slightly --- earlier we have always chosen $\gamma_N \equiv K$ to be a constant, and by a choice of rescaling just set $K=1$, but let us now allow the coupling strength to vary as we change $N$. We will now assume throughout this section that the $\omega_i$ are independent and identically-distributed (iid) random variables, and denote the cumulative distribution function (cdf) of $\omega_i$ as $F(z)$, so that $F(z) := \P(\omega_i \le z)$.  When it exists, we will denote the probability distribution function (pdf) of $\omega_i$ as $f(z) := dF(z)/dz$.  We also assume implicitly below that the $\omega_i$ are not deterministic, i.e. there exists no number $x$ such that $\P(\omega_i = x) = 1$; but if they are then the $\omega_i$ are identical and a phase-locked solution exists trivially.

\begin{define}
Choose and fix a distribution $F(z)$.  Given a constant $\gamma_N$ in~\eqref{eq:gN}, let us define $\pfsync(\gamma_N)$ as the probability that~\eqref{eq:gN} supports a phase-locked solution with $\omega_i$ chosen iid with cdf $F(z)$.  We will typically drop the explicit dependence on $F$ and just write $\psync(\gamma_N)$ when the choice of $F$ is clear.
\end{define}

We can see {\em a priori} that $\psync(\cdot)$ should be a monotone nondecreasing function, as follows.  Let us assume without loss of generality that $\sum_i\omega_i = 0$ (for if not, move to the rotating frame).  Then increasing $\gamma_N$ is equivalent to dilating $\omega_i$ by a decreasing factor but holding $\gamma_N$ constant.  We have shown above that $\mathcal{D}_N$ has the property that dilating $\omega$ can only move it into the phase-locked domain, and not out of it.

The question we address here is how this function depends on $\gamma_N$, and in particular, if there is a natural scaling in which the probability of synchronization goes from zero to one.  It turns out that for a very broad class of distributions, the answer is yes.  As we have seen in the prequel, the boundary of the domain $\mc{D}_N$ is quite complicated, and so {\em a priori} it seems difficult to determine whether or not a random vector will lie in $\mc{D}_N$. However, we can use some of the characterizations of $\mc{D}_N$ above to prove a useful lemma.  Recall Definition~\ref{def:spread}.

\begin{lem}\label{lem:bounds}
For any $N$,
\begin{align*}
\{\omega\in\R^N_0\colon \snorm\omega < \gamma_N\} \subseteq &\mc D_N \subseteq  \{\omega\in\R^N_0\colon \snorm\omega< 2\tau_N\gamma_N/N\} \\&\subseteq  \{\omega\in\R^N_0\colon\snorm\omega< 2\gamma_N\}.
\end{align*}
\end{lem}

\begin{proof}
The first two inclusions follow from the characterization of $\ics N$  and $\ccs N$ in Proposition~\ref{prop:norms} plus some rescaling.
Using the fact that $\tau_N < N$ gives the last inclusion. We can also get the last inclusion directly by the argument
\begin{equation*}
\omega_i-\omega_j =\frac{ \gamma a_N }{N} \sum_{k=1}^N \sin(\theta_i-\theta_k) -\frac{ \gamma a_N }{N} \sum_{k=1}^N \sin(\theta_j-\theta_k)
\end{equation*}
so that $\av{\omega_i-\omega_j} \leq 2 \gamma a_N$.
\end{proof}

It is clear from the above that one quantity of interest will be $\snorm\omega$ when the components of $\omega$ are samples of a particular distribution.  This is known in statistics as the ``sample range'' or ``sample spread'' and is related to extreme value statistics, as we now describe.

\begin{define}
  Let $X_i$ be iid with cdf $F(\cdot)$, and define
\begin{equation*}
  M_N := \snorm{(X_1,\dots, X_N)} = \max_{i=1}^N X_i - \min_{i=1}^N X_i.
\end{equation*}
We say that $X_i$ (or, alternatively, $F(\cdot)$) is {\bf min-max concentrated (MMC)} if there exists a sequence $\xi_N$ such that $M_N/\xi_N\to1$ in probability; more explicitly, this means that for all $\epsilon >0$,
\begin{equation*}
  \lim_{N\to\infty} \P(\av{M_N/\xi_N - 1} > \epsilon) = 0.
\end{equation*}
We will also say that $X_i$ (or $F$) is {\bf MMC($\xi_N$)} for short.
\end{define}

From this we are able to state and prove the main theorem of this section.

\begin{thm}\label{thm:pt}
  Assume that $\omega_i$ are chosen iid with cdf $F(z)$, and that $F(z)$ is MMC($\xi_N)$.  Consider~\eqref{eq:gN} with coupling strength $\gamma_N$. Then
  \begin{equation*}
    \limsup_{N\to\infty} \frac{\gamma_N}{\xi_N} < \frac12 \implies \lim_{N\to\infty}\psync(\gamma_N)  = 0,
  \end{equation*}
  and
   \begin{equation*}
    \liminf_{N\to\infty} \frac{\gamma_N}{\xi_N} >1 \implies \lim_{N\to\infty}\psync(\gamma_N)  = 1.
  \end{equation*}
\end{thm}

\begin{remark}

A few points:

\begin{enumerate}
\item This theorem is saying that there is a phase transition as long as we choose the coupling strength to scale like $\xi_N$; in particular, choosing $\gamma_N = \kappa \xi_N$ with $\kappa>1$ guarantees phase-locking, and with $\kappa<1/2$ guarantees a lack of phase-locking.

\item Note that there is a gap of size $2$ in the statement; if, for example, $\gamma_N/\xi_N\to 3/4$ then we make no claim in this theorem (this gap of size 2 comes from the gap of size two in the previous lemma).

\item It is natural to question which distributions (if any) give rise to random variables that satisfy the assumptions of the theorem; we address this in the remainder of this section.

\end{enumerate}

\end{remark}

{\bf Proof of Theorem~\ref{thm:pt}.}
From Lemma~\ref{lem:bounds} we see that
\begin{equation*}
  \P(\omega\in\mathcal{D}_N) \le \P(M_N \le 2\gamma_N).
\end{equation*}
Let us assume that $\limsup \gamma_N/\xi_N = \kappa < 1/2$.  Then for $N$ sufficiently large, $2\gamma_N \le 2\kappa\xi_N$ and $2\kappa<1$.  Since $\omega_i$ are MCC($\xi_N$), $\P(M_N < 2\kappa\xi_N)\to 0$ which implies $\P(\omega\in\mc{D}_N) \to 0$.  Similarly, we also have
\begin{equation*}
   \P(\omega\in\mathcal{D}_N) \ge \P(M_N \le \gamma_N).
\end{equation*}
If we assume that $\liminf \gamma_N\xi_N  = \kappa > 1$, then for $N$ sufficiently large $\gamma_N \ge \kappa\xi_N$, and $\P(M_N < \kappa \xi_N) \to 1$.
\qed

The next natural question to ask is if there is a more concrete set of assumptions on the distribution $F(\cdot)$ that guarantee the behavior we require from the maximum and minimum of a sample.  This naturally leads into the question of extreme value statistics.

We follow~\cite[Chapter 3]{Coles.book},~\cite[Chapter 2]{Beirlant.book} in our analysis, but give some details here for completeness.  Let $X_i$ be iid with cdf $F(x)$, and define $Q_N = \displaystyle\max_{i=1}^N X_i$.   Then
\begin{equation*}
  \P(Q_N\le z) = \P\left(\max_{i=1}^N X_i \le z\right) = \P(X_i \le z,\ \forall i) = \prod_{i=1}^N \P(X_i\le z) = (F(z))^N =: F^N(z).
\end{equation*}
(Here we have used independence to change the $\forall$ to a product of probabilities.)  Now, if $X_i$ are not essentially bounded (i.e. there does not exist $x^* <\infty$ with $F(z^*) = P(X_i \le x^*) = 1$), then it is clear that $Q_N$ must diverge.  To see this, choose any finite $z$ above; since $F(z) < 1$ then $F^N(z)\to 0$ as $n\to\infty$.  As such, it makes sense to renormalize $Q_N$ in a linear fashion:  $Q_N^* = (Q_N - b_N)/a_N$, where $a_N,b_N$ are sequences that depend on $N$; from the argument above at least one, and perhaps both, of them must diverge as well. Then the famous Fisher--Tippett--Gnedenko Theorem~\cite[Theorem 3.1]{Coles.book} says that if there exists a distribution $G(z)$ such that
\begin{equation}\label{eq:evd}
  \P(Q_N^*\le z) = \P(Q_N \le a_N z + b_N) = (F(a_N z+b_N))^N\to G(z)\quad \mbox{(as $N\to\infty$)},
\end{equation}
then $G(z)$ is in one of three possible distributional families: the Gumbel, Fr\'echet--Pareto, or Weibull families.  These latter are called the {\bf extreme value distributions (EVDs)}, and if~\eqref{eq:evd} holds, we say that $F$ is in the {\bf basin of attraction} of the extreme value distribution $G$.

Moreover, notice that if we choose $Y_i = -X_i$, then $\min X_i = \max Y_i$, and so understanding the minimum of random variables with cdf $F(z)$ is the same as understanding the maximum of random variables with cdf $\widetilde{F}(z) := 1-F(-z)$, and so the theory is the same.

  (Speaking roughly, these three families correspond to the three cases of the random variable being normal-tailed, heavy-tailed, or compactly supported, respectively.  We will make a more precise statement below.)   We are interested here in those random variables that limit into the Gumbel families.  Moreover, notice that if we choose $Y_i = -X_i$, then $\min X_i = \max Y_i$, and so understanding the minimum of random variables with cdf $F(z)$ is the same as understanding the maximum of random variables with cdf $\widetilde{F}(z) := 1-F(-z)$, and so the theory is the same.

\begin{lem}
  Assume that $F(z)$ and $\tilde{F}(z) = 1-F(-z)$ both have the property that it limits onto an extreme value distribution as in~\eqref{eq:evd} with sequences $a_N,b_N$, $\widetilde a_N, \widetilde b_N$ with  $1 \ll a_N \ll b_N$ and $1\ll \widetilde{a}_N \ll \widetilde{b}_N$ as $N\to\infty$.  Then $X_i$ is MMC.
\end{lem}
\begin{proof}
  Let us assume that~\eqref{eq:evd} holds for the maximum $Q_N$ with $1 \ll a_N \ll b_N$.    Choose $\kappa < 1$.  Note that the inequality $\kappa b_N \le a_N z + b_N$ is equivalent to $(\kappa-1) b_N/a_N \le z$.  Note that the assumptions imply that $(\kappa-1) b_N/a_N\to-\infty$, and from this we can find a sequence $z_N\to -\infty$ with
  \begin{equation*}
    \P(Q_N \le \kappa b_N ) \le \P(Q_N \le a_N z_N + b_N) \to G(z_N) \to 0.
  \end{equation*}
Similarly, if $\kappa > 1$, the inequality $\kappa b_N \ge a_N z + b_N$ is equivalent to $(\kappa-1) b_N/a_N \ge z$, and we can find a sequence $z_N \ge    (\kappa-1) b_N/a_N \to \infty$ and
\begin{equation*}
  \P(Q_N \ge \kappa b_N) \ge \P(Q_N \ge a_N z_N + b_N) \to 1-G(z_N) \to 0.
\end{equation*}
In particular, $Q_N/\kappa_N\to1$ in probability.  Clearly the same argument, {\em mutatis mutandis}, holds for the minimum $R_N$ with the parameters $ \widetilde a_N, \widetilde b_N$.  Note however that since $\widetilde{b}_N\gg 1$ by assumption, this implies that $\widetilde{b}_N\to-\infty$.

Now consider the difference $M_N = Q_N-R_N$, where we have that $Q_N/b_N, R_N/\widetilde b_N \to 1$ in probability.   We show that $(Q_N-R_N)/(b_N-\widetilde{b}_N)\to 1$ in probability.  To see this, let $\kappa>1$ and note that $Q_N-R_N > \kappa (b_N-\widetilde{b}_N)$ requires that either $Q_N > \kappa b_N$ or $R_N < \kappa \widetilde b_N$, and (recalling that $\widetilde b_n < 0$),
\begin{equation*}
  \P((Q_N-R_N)/(b_N-\widetilde{b}_N) > \kappa) \le \P(Q_N > \kappa b_N) + \P(R_N < \kappa\widetilde b_N) = \P(Q_N/b_N > \kappa) + \P(R_N/\widetilde b_N > \kappa) \to 0.
\end{equation*}

Since $\widetilde b_N \le b_N$, either ${\widetilde b_N}/b_N \to 0$ or $\limsup {\widetilde b_N}/b_N = \nu \in (-1,1)$.    In the former case, it's clear that $M_N / b_N \to 1$ in probability; in the latter, $M_N/((1-\nu) b_N) \to 1$ in probability.  (Note $\nu > 1$ unless $X_i$ are deterministic.)

\end{proof}

\begin{remark}\label{rem:symmetric}
Finally, one note:  if we assume that the distribution is symmetric, i.e. $F(-x) = 1-F(x)$, then of course we get the same statistics for both the min and the max, but in different directions.  One clever way to get this is break up $X_i = Y_i Z_i$ where $Z_i = \av{X_i}$ and $Y_i = \pm1$, each with probability $1/2$.  If we write $A = \{i:Y_i = 1\}$ and $B = \{i: Y_i=-1\}$, then
\begin{equation*}
  Q_N =  \max_{i\in A} Z_i, \quad R_N = -\max_{i\in B} Z_i.
\end{equation*}
As such, $Q_N$ and $-R_N$ clearly have the same distribution and thus if $Q_N$ is MMC($\xi_N$) then $M_N$ is MMC(2$\xi_N$).
\end{remark}

To summarize,  there are two required conditions:  first, we need that $F$ is in the basin of attraction of an extreme value distribution, but, more importantly, we need to ensure that $b_N \gg a_N$.  Finally, for any given distribution we would need to compute $b_N$ to understanding the correct coupling scaling in~\eqref{eq:gN}.

From here it remains to show that there are some interesting distributions that satisfy the conditions required.  It is also to ask which types of extreme value distributions show up under these conditions.  As it turns out, only the Gumbell can appear here.  For example, if the limiting EVD is Fr\'{e}chet, then we have $b_N = 0$ and $a_N\to\infty$, so the assumptions cannot be satisfied.  If the EVD is Weibull, then again $b_N$ is bounded, and $a_N\to b_N$.  So only the Gumbell will give examples where we obtain the scaling in Theorem~\ref{thm:pt}.  However, a large number of families of well-known distributions lie in the Gumbell class.

\begin{example}
  Let us consider a few examples:

{\bf Gaussian.}  Assume that the $\omega_i$ come from the unit normal distribution $\mathcal{N}(0,1)$.  Then~\cite{hall1979rate, zarfaty2020accurately} the probability in~\eqref{eq:evd} converges to an extreme value distribution with
\begin{equation*}
  a_N = \frac1{\sqrt{2\log(N)}},\quad b_N = \sqrt{2\log(N)} - \frac{\log(4\pi \log(N))}{2\sqrt{2\log(N)}},
\end{equation*}
and we see that
\begin{equation*}
  \frac{b_N}{a_N}  = 2\log(N) - \frac12 \log(4\pi \log(N)) \to \infty.
\end{equation*}

This ensures that we get the phase transition in Theorem~\ref{thm:pt} and we have an explicit expression for $b_N$.

{\bf Exponential and two-sided exponential.}  Let us first assume that $X_i$ is exponential with rate $\lambda > 0$, i.e. $F(z) = 1-\exp(-\lambda z)$ for $z\ge 0$.  Then $\P(Q_N \le z) = (1-\exp(-\lambda z))^N$.  We can first compute
\begin{equation*}
  \E[Q_N] = \int_0^\infty \P(Q_N \ge z)\,dz = \int_0^\infty (1-(1-\exp(-\lambda z))^N\,dz
\end{equation*}
Plugging in $u = 1-\exp(-\lambda z)$ gives us
\begin{equation*}
  \int_0^1 \frac{1-u^N}{\lambda(1-u)}\,du = \frac1\lambda \int_0^1 \sum_{k=0}^{N-1} u^k \,du = \frac1\lambda \sum_{k=1}^{N} \frac 1 k = :\lambda^{-1} H(N),
\end{equation*}
where $H(N)\sim \log(N)$ is the harmonic sum.  This suggests that we should choose $b_N = \log(N)/\lambda$ and $a_N = 1/\lambda$, and note then that
\begin{align*}
  \P((Q_N - b_N)/a_N \le z)  &= \P(Q_N \le a_N z + b_N) =  (1-\exp(-z - \log(N)))^N \\
  &= \left(1 - \frac{e^{-z}}{N}\right) ^N \to e^{-e^{-z}}.
\end{align*}
It is a commonly used fact for Markov chains~\cite{Norris.book} that $R_N = \min_{i=1}^{N} X_i$ has exponential distribution with rate $\lambda N$, and as such $R_N\to 0$ w.p.1 and certainly in probability.  Thus $Q_N$ and $M_N$ have the same distribution.

We can also consider the two sided exponential with distribution $F_2(z) = (1/2) (F(z) + (1-F(-z)))$, which we can again think of choosing a coin flip $Y_i$ for a sign $\pm1$, and $Z_i$ to be an exponential, and then $X_i = Y_iZ_i$.  Here we can choose $b_N = 2\log(N/2)/\lambda$ instead.

\end{example}

In fact, a quite general sufficient condition that guarantees Gumbel convergence is:  if there exists an auxiliary function $b(x)$ and
\begin{equation*}
  \lim_{x\to\infty} \frac{1-F(x+tq(x))}{1-F(x)} \to e^{-t},\mbox{ for all } t>0,
\end{equation*}
then $F$ limits into the Gumbel class.  (This is a part of full statement of the Fisher--Tippett--Gnedenko Theorem, see~\cite[Theorem 2.1]{Beirlant.book}.)  More concretely, if $F$ is in this basin of attraction, (see~\cite[Section 3.3.3]{Embrechts.book}) we can always define
\begin{equation*}
  q(t) := \frac{\int_t^\infty 1-F(s)\,ds}{1-F(t)},\quad b_N = F^{-1}(1-1/N),\quad a_N = q(b_N).
\end{equation*}
So, as long as $q(t)$ has the property that $\lim_{t\to\infty} q(t)/t = 0$, then this will apply.

\section{Conclusions}
We have given a unified treatment of necessary and sufficient
conditions on the frequency vector $\omega$ for the existence of a
stably phase-locked solution to the Kuramoto system. This construction
gives a necessary condition that is dual to the well-known
D\"{o}rfler--Bullo sufficient condition, and likewise a sufficient
condition dual to the Chopra--Spong necessary condition. Both of these
conditions are new, and the first is (for four or more oscillators) a
sharper condition than previously known conditions in the sense of the
$(N-1)$ dimensional Euclidean volume. Moreover we have shown how to
combine two norm estimates to get a new estimate that is better than
either one. This construction gives us further new conditions that
improve on those in the literature; the sufficient condition strictly contains
the D\"{o}rfler--Bullo sufficient condition and the necessary condition is strictly
contained in the Chopra--Spong necessary condition.

We also established a probabilistic phase-locking result for very
general distributions of natural frequencies.
We used the fact that the range semi-norm is equivalent to the
Kuramoto semi-norm, along with some known facts about the limiting
distribution of extreme value statistics to prove the following dichotomy:
for coupling strengths below a certain threshold complete
phase-locking occurs with probability zero, while above a (different)
threshold complete phase-locking occurs with probability one. This is
a substantial generalization of the results of Bronski, DeVille and
Park \cite{BDP}.

A few comments on possible generalizations and their nontriviality.  There are some natural generalizations that come to mind:  one could generalize the $\sin(\cdot)$ coupling in the paper to more general coupling functions (i.e.~generalize from the Kuramoto system to the Kuramoto--Sakaguchi or Kuramoto--Daido models).  It is also natural to ask what happens if we generalize to models on an arbitrary graph (i.e.~remove the ``all-to-all'' coupling and choose a more general coupling based on an underlying graph).  
Two of the main ingredients used above are {\bf symmetry} of the Jacobian of the vector field and {\bf convexity} of the set $\mathcal{D}_N$.    The symmetry of the Jacobian will hold generally only when the coupling function is odd (or, more specifically, when its derivative is even).  Of course, even in the Kuramoto--Sakaguchi model, there will be configurations $\theta$ for which the Jacobian is symmetric, but this will not hold in general.  The derivation of Lemma 2.2 (which is~\cite[Lemma 4.4]{BDP}) strongly uses symmetry of the Jacobian by  a Courant minimax argument.  We can no longer give a general ``if and only if'' condition when the Jacobian is no longer symmetric and there is no general method here.  One might then ask if all we require that the coupling function be odd, why restrict to $\sin(\cdot)$ but instead use a general odd function, e.g. something like a Fourier sine series with more than one frequency?  The issue here then arises that the stability domain is no longer required to be convex, and informal numerical results by the authors show that for some odd coupling functions the stability domain loses its convexity.  Finally, the question about networks is quite natural, but again we run into a challenge.  It is not clear how to show that for a general network coupling the domain $\mathcal{D}_N$ is convex, or (probably more likely) how to characterize which networks give a convex domain.  This is beyond the scope of this paper but there are clearly several interesting questions worthy of further study here.

\section{Acknowledgements}
JCB would like to acknowledge support under National Science Foundation grant NSF-DMS 1615418. TEC would like to acknowledge support from Caterpillar Fellowship Grant at Bradley University.  The authors would like to thank the anonymous referees whose comments and suggestions greatly improved the final version of this article.

%

\appendix

\section{Example of stability with negative entries}\label{sec:example-negative}

Here we give a few families of interesting examples of stable points, as promised in Remark~\ref{rem:ermentrout}.

Let us write $\nu_{ij} = \cos(\theta_i-\theta_j)$.  Then \cite[Theorem 3.1]{Ermentrout.1992} says that if
\be
\ii $\nu_{ij} \ge 0$ for all $i\neq j$;
\ii for any $k_0, k_m\in\{1,\dots, n\}$, there exists a path $k_0, k_1,\dots, k_m$ with $\nu_{k_{q-1},k_q} > 0$ for all $q=1,\dots,m$;
\ee
then the configuration is stable.    The theorem does not speak to what happens if some of the $\nu_{ij}$ are negative.  In contrast, our Lemma~\ref{lem:stability} gives a restriction only on the sum of the $\nu_{ij}$, which would allow for $\nu_{ij}$ of both signs.  Here we show that it is possible to obtain stable configurations with some $\nu_{ij}$ strictly negative.

As a concrete example, let us consider the following one parameter family $\theta_\alpha = \{-\alpha,0,\alpha\}$.  We claim that for $\pi/4 < \alpha < 0.297916 \pi$, the assumption of Ermentrout's Theorem is false, but the assumptions in our Lemma are true, and in fact the Jacobian is stable.  Note here that we have
\begin{equation*}
  \nu_{11}(\alpha)=\nu_{22}(\alpha)=\nu_{33}(\alpha) = 1,\quad \nu_{12} (\alpha)= \nu_{21}(\alpha) =\nu_{23}(\alpha)=\nu_{32} (\alpha)= \cos\alpha,\quad \nu_{13}(\alpha)=\nu_{31}(\alpha) = \cos2\alpha.
\end{equation*}
It is easy to see that $\nu_{13}(\alpha),\nu_{31}(\alpha)$ become negative as $\alpha$ passes through $\pi/4$.  However, we have
\begin{equation*}
  \kappa_1(\alpha) = \kappa_3(\alpha) = 1+ \cos\alpha+\cos2\alpha,\quad \kappa_2(\alpha) = 1+2\cos\alpha,\quad \tau(\alpha) = \frac{2}{\cos \alpha+\cos 2 \alpha +1}+\frac{1}{2 \cos \alpha +1}
  \end{equation*}
Now, to solve the equation $\tau(\alpha) =2$, we first use the trig identity $\cos2\alpha = 2\cos^2\alpha-1$.  Writing $\beta = \cos\alpha$ and some manipulation gives the quadratic $4\beta^2 + \beta - 2-0$ with roots $\beta_\pm = (1/8)(1\pm\sqrt{33})$.  Noting that both roots are in $[-1,1]$, gives four roots at $\alpha_\pm = \cos^{-1}((1/8)(1\pm\sqrt{33})$ and $2\pi-\alpha_\pm$.  Noting that $\tau$ is even (indeed, the original problem is clearly invariant under $\alpha\mapsto\pm\alpha$) means that we need to look for the first positive root, and therefore we are looking for
\begin{equation*}
  \alpha_+ = \cos^{-1} \left(\frac{-1+\sqrt{33}}8\right)\approx 0.935929 = 0.297916\pi.
\end{equation*}
It is straightforward to check that $\kappa_i(\alpha)$ are all decreasing on $(0,\alpha_+)$ and
\begin{equation*}
  \kappa_1(\alpha_+) = \kappa_3(\alpha_+) = \frac{15+\sqrt{33}}{16},\quad \kappa_2(\alpha_+) = \frac{3+\sqrt{33}}4,
\end{equation*}
which are both positive.  Therefore the configuration is stable for all $\alpha \in [0,\alpha_+)$, but some of the $\nu$ will be negative for ${\pi/4} < \alpha < \alpha_+$; specifically, we see that $\nu_{13}(\alpha_+) \approx -0.296535$.  We note in passing here that this family of solutions is related to the Chopra--Sprong points $\rcs3$ from earlier; in fact plugging in the point $(\alpha_+,0,-\alpha_+)$ into the nonlinearity on the right-hand side of the Kuramoto equation gives one of the Chopra--Spong points exactly (we can obtain the other five by taking negatives and permutations).

\section{Proof of Proposition~\ref{prop:volume}}\label{sec:Postnikov}
In this section we outline the derivation of the volume of the
polytope $P(R_n^{\text{CS}})$. We will compute the volume of the
unit permutathedron whose vertices are all permutations of
$(1,0\ldots,0,-1)$ -- the volume of $P(R_n^{\text{CS}})$ will obviously be
$\tau_N^{N-1}$ times the volume of the unit permutahedron. Postnikov\cite{Postnikov.2009}
has given several formulae for the $(N-1)$-volume of a permutahedron-- the
polyhedron whose vertices are given by all permutations of
$(x_1,x_2,\ldots,x_{N})$ which clearly lies in the $N-1$ dimensional
hyperplane in ${\mathbb R}^{N}$ given by $\sum x_i=\text{constant}.$
Of these formulae perhaps the most straightforward to apply in this instance
is Theorem 3.2, which expresses the volume of the permutahedron
$P_n(x_1,x_2,\ldots,x_N)$ as a polynomial in $x_i:$
\begin{equation}
  \vol(P_N) = \sum (-1)^{|I_{c_1 c_2\ldots c_N}|} D_N(I_{c_1 c_2
    \ldots c_N} ) \frac{x_1^{c_1}}{c_1!} \frac{x_2^{c_2}}{c_2!}\ldots \frac{x_N^{c_N}}{c_N!}.
\label{eqn:PostnikovMagic}
\end{equation}
Here the sum is over all sequences of non-negative integers $c_i$ such
that $\sum c_i= N-1$, $I$ is a certain set of integers $i \in
\{1,2,\ldots N-1\}$ to be defined shortly, and $D_N(I)$ is the number
of permutations with descent
set $I$. In particular given a sequence $\{c_i\}_{i=1}^N$ with
$\sum c_i=N-1$ one first defines a sequence $\epsilon \in
\{-1,1\}^{2(N-1)}$ by the following rule: each $c_i$ in the original
sequence contributes to $\epsilon$ $c_i$ ``1'''s followed by a single
$-1$, with the last $-1$ being deleted. For instance
$(c_1,c_2,c_3,c_4) = (2,0,0,1)$ gives $\epsilon =
(1,1,-1,-1,-1,1)$. The set $I_{c_1 c_2\ldots c_N}$ is defined to be
$\{i \in \{1,2\ldots,N-1\}| \sum_{j=1}^{2i-1} \epsilon_j
<0\}$. Finally $D_N(I)$ is defined to be the number of permutations
having descent set $I$, where the descent set of a permutation $\sigma
\in S_n$ is $\{i \in \{1,2\ldots N-1| \sigma(i) > \sigma(i+1)\}$.

Formula (\ref{eqn:PostnikovMagic}) is particularly nice in our case
since we have $x_1=1,x_N=-1$ and the remaining $x_j=0$. Thus the only
terms that contribute are those where $c_1=j$ and $c_N=(N-1-j)$, with
the remaining $c_k=0$, for $j \in \{0,1,\ldots,N\}$. The sequence
$\epsilon$ consists of $j$ 1's, followed by $(N-1)$ $-1$'s, followed by
$(N-j-1)$ $1$'s. The set $I_{j 0 0 \ldots 0 (N-j-1)}$ is clearly
$\{(j+1),(j+2),\ldots(N-1)\}$. Next we need to count the number of
permutations of $\{1\ldots N\}$ that have descent set $\{(j+1),(j+2),\ldots(N-1)\}$ --- in
other words permutations that are increasing up to $(j+1)$ and
decreasing after.  It is not hard to see that there are
${{N-1}\choose{j}}$. To see this note that the largest element, $n$,
must occur at position $j+1$. One can choose $j$ elements from
$\{1\ldots N-1\}$ to occur in the first $j$ positions. They must, of
course, be in increasing order with the remaining $N-1-j$ elements in
the last $N-j-1$ positions in decreasing order.    This gives
\begin{align}
  \vol(P_N(1,0,\ldots,0,-1))& = \sum_{j=0}^{N-1} \frac{(N-1)!}{j! j!
    (N-1-j)!(N-1-j)!}& \\
  & = \frac{1}{(N-1)!} \sum_{j=0}^{N-1}
    {{N-1}\choose{j}}^2& \\
    & = \frac{1}{(N-1)!} {{2(N-1)}\choose{N-1}} &,
  \end{align}
 where the last line follows from the well-known combinatorial
 identity $\sum_{j=0}^N{{N}\choose{j}}^2 = {{2N}\choose{N}}$.

There is a minor additional multiplicative factor to consider: the
Postnikov result is normalized so that the volume of a fundamental
cell of the lattice is one (equivalently it is the volume of the
projection of the polytope $p(P_n)$, where $p:(t_1,t_2,\ldots,t_N)
\mapsto (t_1,t_2,\ldots,t_{N-1}$). The generators of the lattice are
$v_1 = (1,0,\ldots,0,-1), v_2 = (0,1,0\ldots 0,-1), \ldots , v_{N-1}=
(0,\ldots,0,1,-1)$. The usual Euclidean $(N-1)$-volume of the fundamental cell
is given by $\sqrt{\det(G)}$ where $G$ is the Gram matrix $G_{ij}=v_i
\cdot v_j$. It is easy to see that, given the generators above the
Gram matrix  is
\[
G_{ij} = \begin{cases} 2, & i=j\\
1& i\neq j\end{cases}, 
\]
where $i,j$ range over $1,\dots, N-1$. It is also easy to see that $\det(G)=N$ (The eigenvalues are $1$, with
multiplicity $(N-2)$ and $N$ with multiplicity $1$). Thus it follows
that
\[
\vol(P_{R_n^{\text{CS}}}) = \frac{\sqrt{N}}{(N-1)!} {{2(N-1)}\choose{(N-1)}} \tau_N^{N-1}.
\]

\section{Proof of Proposition~\ref{prop:norms}}

\begin{proof}

Throughout the proofs when considering the points $\rcs N$ (resp.~$\rcs{N,j}$) it will be more convenient to scale out the
factor of $\tau_N$ (resp.~$\tau_{N,j}$) and work with vectors with integer entries.  Also recall Notation~\ref{not:abc} for the arguments below.

\fbox{{\bf  Case 1:} $\idb N$.}  Recall the definition of $\rdb N$:
\begin{equation*}
  \rdb N = \bigcup_{k=1}^N \Sym\left(\left((N-j)^{(j)}, (-j)^{(N-j)}\right)^t\right)
\end{equation*}
 It is clear that $\max_{i,j} (v_i - v_j) = N$ for any vector in $\rdb N$. From the
triangle  inequality it is clear that for any convex linear
combination of such vectors $w$ we have that  $\max_{i,j} (w_i -
w_j) \leq N.$ So the polytope includes the ball $\max_{i,j} (w_i -
w_j) \leq N.$ If we show that every point on the boundary of the polytope has $\max_{i,j} (w_i -
w_j) = N$, then it follows that the polytope given by convex
combinations of vertices  is exactly the ball of
radius $N$: $\max_{i,j} (w_i -
w_j) \leq N$ .

The polytope with vertices $R^{\text{DB}}_N $ has $N(N-1)$ faces $F_{k,l}$ with $1\le
k\neq l\le N$, which we describe now. Choose $k,l$, and define
$\tilde R$ to be the subset of vectors in $R^{\text{DB}}_N$ such that
the $k^{th}$ component is positive and the  $l^{th}$ component is negative.
It is easy to see that if $w$ is any convex combination of the
vectors in $\tilde R$ then
\begin{itemize}
\item Component $k$ is the largest positive component (possibly not
  unique.)
\item Component $l$ is the most negative component (possibly not
  unique.)
\item $\max_{ij} (w_i-w_j) = N.$
\end{itemize}
To see that this is a face note that all of the vectors in $\tilde R$,
and thus any convex combination of them, lie in the plane $w
\cdot (e_k - e_l)=N$, and that adding any positive multiple of
the vector $e_k - e_l $ to $w$ results in a vector
that has $\max_{ij} (w_i - w_j) > N$ and thus is not in the polytope.
(Alternatively one can also use the fact that $(e_k - e_l)$ is the normal
to the face together with Lemma~\ref{lem:Normals}.)

\fbox{{\bf  Case 2:} $\ics N$.}  We show that the set of all convex  combinations of permutations of the
vector $(1^{(1)}, (-1)^{(1)}, 0^{N-2})^t$ is contained in and contains the  $L_1$ unit
ball $\{y \in \R^N_0 : \norm y_1  \leq
2 \} $. The result then follows from scaling.

One direction is easy: the vectors all have $L_1$ norm equal to $2$, and thus any convex combination will have $L_1$ norm less than or equal to $2$, so the polytope is contained in the $L_1$ ball of radius $2$.

To see the other direction we give an explicit ``greedy''  algorithm to decompose
any vector in the ball of radius $2$ into a convex combination of vectors of
the given form. We can assume without loss of generality that the
given vector has $L_1$ norm equal to $2$. Given a vector $v$ with
$\Vert v\Vert_1=\beta $ we define $i$ to be the component of $v$ with
the smallest non-zero magnitude. (If there are multiple such
components any one can be chosen). Let $j$ be any component with the
opposite sign of component $i$. Consider the new vector $v \pm
|v_i| (e_i- e_j),$ where the sign is chosen so that
the $i^{th}$ components cancel.  It is easy to see that this
operation has the following properties.
\begin{itemize}
\item It zeroes out component $i$.
\item It decreases the magnitude of component $j$. Component $j$ may be
  zero but it cannot change sign.
\item It decreases the $L_1$ norm by exactly $2|v_i|$.
\item It leaves the remaining components unchanged.
\end{itemize}
It is easy to see that this algorithm terminates in at most $N-1$
steps. Since the initial $L_1$ norm is $2$ and the decrease in the
$L_1$ norm at each step is twice the coefficient the coefficients sum
to $1$. Thus every vector with $L_1$ length 2 is expressible as a
convex combination of the basis vectors and lies in the closed
polytope. Since the polytope is contained in and contains the $L_1$
ball of radius $2$ the two must be the same.

\fbox{{\bf  Case 3:} $\ics{N,j}$.} This follows more or less directly from Rado's theorem.
In our case $v$ is the vector $(1^{(j)}, 0^{(N-2j)}, (-1)^{(j)})^t$ and the permutahedron is given by the set of vectors $y$ which satisfy
the following set of inequalities.
\begin{align*}
&y_1 \leq 1& \\
&y_1 + y_2 \leq 2 & \\
&y_1 + y_2 + y_3 \leq 3 &\\
& \vdots& \\
&y_1 + y_2 + \ldots + y_j \leq j&\\
&y_1 + y_2 + \ldots + y_{j+1} \leq  j&\\
&y_1 + y_2 + \ldots + y_{j+2} \leq  j&\\
& \vdots& \\
&y_1 + y_2 + \ldots + y_{N-j} \leq j&\\
&y_1 + y_2 + \ldots + y_{N-j+1} \leq  j-1&\\
& \vdots& \\
& y_1 + y_2 + \ldots + y_{N-2} \leq  2&\\
& y_1 + y_2 + \ldots + y_{N-1} \leq  1&
\end{align*}
The first inequality implies that the largest entry is less than or equal to
$1$. The last inequality (together with the condition that $\sum y_i=0$ )
implies that $-y_N\leq 1$. These together imply that $\Vert y\Vert_{\infty}\leq 1$. Next we note that $2(j-1)$ of these inequalities are redundant: given
that the first inequality holds it follows from the ordering of the $y_i$
that the second through the $j^{th}$ must also hold. The remaining $N-2j$
inequalities are of the same form,
\[
\sum_{i=1}^k y_i \leq j \qquad\qquad k \in (j+1, N-j-1).
\]
The sum $\sum_{i=1}^k y_i$ is going to be maximized by some $k^*$ (possibly non-unique) and the inequalities  hold if and only if the inequality holds for this value of $k$. The sum is obviously maximized when the summation
contains all of the positive terms and none of the negative terms (and the
disposition of any zero terms does not matter). Since the terms $y_i$ sum
to zero the sum of the positive terms is equal to minus the sum of the
negative terms, and thus each is equal to $\frac{1}{2} \sum_i |y_i|$.
Thus all of the inequalities above can be reduced to two conditions $\norm y_\infty \le 1$ and $\norm y_1 \le 2$.
Thus the polytope is defined by the condition $ \max\left(2\norm y_\infty,{j}^{-1}\norm y_1\right) \leq 2$.

The case $j=1$ reduces nicely:  since we are working on mean zero space $\sum_i y_i =0 $ we have that $\norm y_1 \ge 2\norm y_\infty,$ recovering the previous formula.

\fbox{{\bf  Case 4:} $\cdb N$.}
In this case the normals to the faces of the polytope are given by all
permutations of all vectors of the form $ (i^{(j)}, (-j)^{(i)})^t$ for $i \in
\{1\ldots N-1\}$ and $j=N-i$. (Recall Notation~\ref{not:abc}.)
 From the rearrangement inequality we
can assume that that the entries of $y$ and the entries of the normal
vector are both arranged in decreasing order. Then  Lemma
\ref{lem:Normals} implies
\[
 \max_{i} \frac{j \sum_{l=1}^i y_l - i \sum_{l=i+1}^N y_l}{i j N } \leq 1
\]
Note that the vector $y$ has mean zero and thus $\sum_{l=i+1}^N y_l
=-\sum_{l=1}^i y_l$. Thus we have
\begin{align*}
& \max_{i} \frac{j \sum_{l=1}^j y_l - i \sum_{l=k+1}^N y_l}{i j N} \leq 1& \\
&\max_i \frac{(i+j) \sum_{l=1}^k y_l}{i j N} \leq 1& \\
&\max_i \frac{ \sum_{l=1}^i y_l}{i j } \leq 1&
\end{align*}

\fbox{{\bf Case 5:} $\ccs N$.}
In this case the normal vectors to the faces of the polytope are given
by all permutations of $\tau_N (1^{(1)}, (-1)^{(1)}, 0^{N-2})^t$, which can be written as $\tau_N (e_i - e_j)$, where $e_i$ is the unit vector in the $i^{th}$ corrdinate direction.
 From Lemma \ref{lem:Normals} it follows that the polytope is defined by $\max_{i,j} \frac{\tau_N (e_i - e_j) \cdot y}{\tau_N(e_i - e_j) \cdot \tau_N (e_i - e_j) } = \max_{i,j} \frac{\tau_N(y_i-y_j)}{2 \tau_N^2} \leq 1$, or $y_{\max} - y_{\min} \leq 2 \tau_N$.

\fbox{{\bf Case 6:} $\ccs{N,j}$.} This case is similar to the above and the result again follows  from Lemma~\ref{lem:Normals}.  The normals are given by all permutations of $\tau_{N,j} (1^{(j)},(-1)^{(j)},0^{(N-2j)}).$ The squared Euclidean length of any normal is $\Vert x \Vert^2= 2 j \tau_{N,j}^2. $ It is clear that the quotient $\frac{\langle y,x\rangle}{\langle x,x\rangle}$ is maximized
when the $j$ components of $x$ equal to $+1$ correspond to the $j$ largest components of $y$, and likewise the $j$ components of $x$ equal to $-1$ correspond to the
$j$ smallest components of $y$. This gives the condition
\[
\frac{\sum_{i=1}^j y_{\text{max},i} - \sum_{i=1}^j y_{\text{min},i}}{2 j \tau_{N,j}} \leq 1.
\]
The intersection of these inequalities is obviously given by
\[
\max_j \frac{\sum_{i=1}^j y_{\text{max},i} - \sum_{i=1}^j y_{\text{min},i}}{2 j \tau_{N,j}} \leq 1.
\]

\end{proof}

\end{document}